\documentclass{amsart}
\usepackage[utf8]{inputenc}
\usepackage{amssymb}
\usepackage{hyperref}
\usepackage[final]{showkeys} 

\input xy
\xyoption{all}

\theoremstyle{definition}
\newtheorem{mydef}{Definition}[section]
\newtheorem{lem}[mydef]{Lemma}
\newtheorem{thm}[mydef]{Theorem}

\newtheorem{cor}[mydef]{Corollary}

\newtheorem{defin}[mydef]{Definition}

\newtheorem{remark}[mydef]{Remark}

\newtheorem{notation}[mydef]{Notation}
\newtheorem{fact}[mydef]{Fact}

\newcommand{\fct}[2]{{}^{#1}#2}

\newcommand{\ba}{\bar{a}}
\newcommand{\bb}{\bar{b}}

\newcommand{\Ksatpp}[2]{{#1}^{#2\text{-sat}}}
\newcommand{\Ksatp}[1]{\Ksatpp{\K}{#1}}

\newcommand{\dom}{\operatorname{dom}}
\newcommand{\ran}{\operatorname{ran}}

\newcommand{\cf}[1]{\text{cf} (#1)}
\newcommand{\seq}[1]{\langle #1 \rangle}
\newcommand{\rest}{\upharpoonright}

\newcommand{\s}{\mathfrak{s}}

\newcommand{\ts}{\mathfrak{t}}

\newcommand{\leap}[1]{\le_{#1}}
\newcommand{\ltap}[1]{<_{#1}}

\newcommand{\geap}[1]{\ge_{#1}}

\newcommand{\lta}{\ltap{\K}}
\newcommand{\lea}{\leap{\K}}

\newcommand{\K}{\mathbf{K}}

\newcommand{\Pp}[1]{\mathcal{P}_{#1}}
\newcommand{\PK}{\Pp{\K}}
\newcommand{\PKp}[1]{\Pp{\K_{#1}}}

\newbox\noforkbox \newdimen\forklinewidth
\forklinewidth=0.3pt \setbox0\hbox{$\textstyle\smile$}
\setbox1\hbox to \wd0{\hfil\vrule width \forklinewidth depth-2pt
 height 10pt \hfil}
\wd1=0 cm \setbox\noforkbox\hbox{\lower 2pt\box1\lower
2pt\box0\relax}
\def\unionstick{\mathop{\copy\noforkbox}\limits}
\newcommand{\nf}{\unionstick}
\newcommand{\nfs}[4]{#2 \nf_{#1}^{#4} #3}

\def\1nf{\unionstick^{(1)}}

\def\2nf{\unionstick^{(2)}}
\def\3nf{\unionstick^{(3)}}

\newcommand{\gtp}{\mathbf{tp}}
\newcommand{\gtpp}[1]{\gtp_{#1}}
\newcommand{\gS}{\mathbf{S}}

\newcommand{\Ii}{\mathbb{I}}

\newcommand{\ehanf}[1]{\beth_{\left(2^{#1}\right)^+}}

\newcommand{\Ll}{\mathbb{L}}

\newcommand{\tlt}{\triangleleft}
\newcommand{\tleq}{\trianglelefteq}

\newcommand{\goodp}{\text{good}^+}

\newcommand{\LS}{\text{LS}}

\title{Tameness from two successive good frames}
\date{\today \\
AMS 2010 Subject Classification: Primary 03C48. Secondary: 03C45, 03C52, 03C55, 03C75.}
\keywords{abstract elementary classes, good frames, tameness}

\author{Sebastien Vasey}
\email{sebv@math.harvard.edu}
\urladdr{http://math.harvard.edu/\textasciitilde sebv/}
\address{Department of Mathematics \\ Harvard University \\ Cambridge, Massachusetts, USA}

\begin{document}

\begin{abstract}
  We show, assuming a mild set-theoretic hypothesis, that if an abstract elementary class (AEC) has a superstable-like forking notion for models of cardinality $\lambda$  and a superstable-like forking notion for models of cardinality $\lambda^+$, then orbital types over models of cardinality $\lambda^+$ are determined by their restrictions to submodels of cardinality $\lambda$. By a superstable-like forking notion, we mean here a good frame, a central concept of Shelah's book on AECs. 

  It is known that locality of orbital types together with the existence of a superstable-like notion for models of cardinality $\lambda$ implies the existence of a superstable-like notion for models of cardinality $\lambda^+$, and here we prove the converse. An immediate consequence is that forking in $\lambda^+$ can be described in terms of forking in $\lambda$.
\end{abstract}

\maketitle

\tableofcontents

\section{Introduction}

Good frames are the central notion of Shelah's two volume book \cite{shelahaecbook, shelahaecbook2} on classification theory for abstract elementary classes (AECs). Roughly speaking, an AEC is a concrete category (whose objects are structures) satisfying several axioms (for example, morphisms must be monomorphisms and the class must be closed under directed colimits). It generalizes the notion of an elementary class (i.e.\ a class axiomatized by an $\Ll_{\omega, \omega}$-theory, where the morphisms are elementary embeddings) and also encompasses many infinitary logics such as $\Ll_{\infty, \omega}$ (i.e.\ disjunctions and conjunctions of arbitrary, possibly infinite, length are allowed). An AEC $\K$ has a good $\lambda$-frame if its restriction to models of cardinality $\lambda$ is reasonably well-behaved (e.g.\ has amalgamation, no maximal models, and is stable) and it admits an abstract notion of independence (for orbital types of elements over models of cardinality $\lambda$) that satisfies some of the basic properties of forking in a superstable elementary class: monotonicity, existence, uniqueness, symmetry, and local character. Here, local character is described not as ``every type does not fork over a finite set'' but as ``every type over the union of an increasing continuous chain of models of cardinality $\lambda$ does not fork over some member of the chain''. The theory of good frames is used heavily in several recent results of the classification theory of AECs, including the author's proof of Shelah's eventual categoricity conjecture in universal classes \cite{ap-universal-apal, categ-universal-2-selecta}, see also \cite{bv-survey-bfo} for a survey.

The reason for restricting oneself to models of cardinality $\lambda$ is that the compactness theorem fails in general AECs, and so it is much easier in practice to exhibit a local notion of forking than it is to define forking globally for models of all sizes. In fact, Shelah's program is to start with a good $\lambda$-frame and only then try to extend it to models of bigger sizes. For this purpose, he describes a dividing line\footnote{The nonstructure side is described by \cite[II.5.11, II.8.4, II.8.5]{shelahaecbook}, showing that failure of successfulness implies the AEC has many non-isomorphic models.}, being successful, and shows that if a good $\lambda$-frame is successful, then there is a good $\lambda^+$-frame on an appropriate subclass of $\K_{\lambda^+}$.

A related approach is to outright assume some weak amount of compactness. Tameness \cite{tamenessone} was proposed by Grossberg and VanDieren to that end: $\lambda$-tameness says that orbital types are determined by their restrictions of cardinality $\lambda$. This is a nontrivial assumption, since in AECs syntactic types are not as well-behaved as one might wish, so one defines types purely semantically (roughly, as the finest notion of type preserving isomorphisms and the $\K$-substructure relation). It is known that tameness follows from a large cardinal axiom (see Fact \ref{tamelc}) and some amount of it can be derived from categoricity (see Fact \ref{tameness-categ}). The present paper gives another way to derive some tameness.

Grossberg conjectured in 2006 that\footnote{See the introduction of \cite{jarden-tameness-apal} for a detailed history.}, assuming amalgamation in $\lambda^+$, a good $\lambda$-frame extends to a good $\lambda^+$-frame if the class is $\lambda$-tame. He told his conjecture to Jarden who could prove all the axioms of a good $\lambda^+$-frame, except symmetry. Boney \cite{ext-frame-jml} then proved symmetry assuming a slightly stronger version of tameness, Jarden \cite{jarden-tameness-apal} proved symmetry from a certain continuity assumption, and Boney and the author finally settled the full conjecture \cite{tame-frames-revisited-jsl}, see Fact \ref{ext-frame}. In this context, forking in the good $\lambda^+$-frame can be described in terms of forking in the good $\lambda$-frame. Let us call this result the \emph{upward frame transfer theorem}.

This paper discusses the converse of the upward frame transfer theorem. Consider the following question: \emph{if} there is a good $\lambda$-frame and a good $\lambda^+$-frame, can we say anything on how the two frames are related (i.e.\ can forking in $\lambda^+$ be described in terms of forking in $\lambda$?) and can we conclude some amount of tameness? We answer positively by proving the following converse to the upward frame transfer theorem:

\textbf{Corollary \ref{main-cor-simplified}.}
Let $\K$ be an AEC and let $\lambda \ge \LS (\K)$. Assume $2^{\lambda} < 2^{\lambda^+}$. If there is a categorical good $\lambda$-frame $\s$ on $\K_{\lambda}$ and a good $\lambda^+$-frame $\ts$ on $\K_{\lambda^+}$, then $\K$ is $(\lambda, \lambda^+)$-tame. Moreover, $\s_{\lambda^+} \rest \Ksatp{\lambda^+}_{\lambda^+} = \ts \rest \Ksatp{\lambda^+}_{\lambda^+}$ (see Definition \ref{s-lambdap-def}).

In the author's opinion, Corollary \ref{main-cor-simplified} is quite a surprising result since it shows that we cannot really study superstability in $\lambda$ and $\lambda^+$ ``independently'': the two levels must in some sense be connected (this follows from a canonicity result for good frames, see the paragraph after next). Put another way, two successive local instances of superstability already give a nontrivial amount of compactness. In fact after the initial submission of this article, Corollary \ref{main-cor-simplified} was used by the author \cite{categ-amalg-v4} as a key tool to \emph{prove} a class has some amount of tameness and deduce strong versions of Shelah's eventual categoricity conjecture in several types of AECs.

In Corollary \ref{main-cor-simplified}, ``categorical'' simply means that $\K$ is assumed to be categorical in $\lambda$. We see it as a very mild assumption, since we can usually restrict to a subclass of saturated models if this is not the case, see the discussion after Definition \ref{sl-def}. As for $(\lambda, \lambda^+)$-tameness, it means that types over models of cardinality $\lambda^+$ are determined by their restrictions of cardinality $\lambda$. In fact, it is possible to obtain a related conclusion by starting with a good $\lambda^+$-frame on the class of \emph{saturated} models in $\K$ of cardinality $\lambda^+$. In this case, we deduce that $\K$ is $(\lambda, \lambda^+)$-\emph{weakly} tame, i.e.\ only types over \emph{saturated} models of cardinality $\lambda^+$ are determined by their restrictions to submodels of cardinality $\lambda$. We deduce that weak tameness is equivalent to the existence of a good $\lambda^+$-frame on the saturated models of cardinality $\lambda^+$, see Corollary \ref{main-cor-gch}.

An immediate consequence of Corollary \ref{main-cor-simplified} is that forking in $\lambda^+$ (at least over saturated models) can be described in terms of forking in $\lambda$. Indeed, the upward frame extension theorem gives a good $\lambda^+$-frame with such a property, and good frames on subclass of saturated models are canonical: there can be at most one, see Fact \ref{canon-fact}. In fact, assuming that forking in $\lambda^+$ is determined by forking in $\lambda$ is equivalent to tameness (see \cite[3.2]{ext-frame-jml}) because of the uniqueness and local character properties of forking. In Corollary \ref{main-cor-simplified} we of course do \emph{not} start with such an assumption: forking in $\lambda^+$ is \emph{any} abstract notion satisfying some superstable-like properties for models of cardinality $\lambda^+$.

The proof of Corollary \ref{main-cor-simplified} goes as follows: we use $2^{\lambda} < 2^{\lambda^+}$ to derive that the good $\lambda$-frame is \emph{weakly successful} (a dividing line introduced by Shelah in Chapter II of \cite{shelahaecbook}). This is the only place where $2^{\lambda} < 2^{\lambda^+}$ is used. Being weakly successful imply that we can extend the good $\lambda$-frame from types of singletons to types of models of cardinality $\lambda$. We then have to show that the good $\lambda$-frame is also successful. This is equivalent to a certain reflecting down property of nonforking of models. Jarden \cite{jarden-tameness-apal} has shown that successfulness follows from $(\lambda, \lambda^+)$-weak tameness and amalgamation in $\lambda^+$, and here we push Jarden's argument further by showing that having a good $\lambda^+$-frame suffices, see Theorem \ref{main-thm}. A key issue that we constantly deal with is the question of whether a union of saturated models of cardinality $\lambda^+$ is saturated. In Section \ref{decent-sec}, we introduce a new property of forking, being decent, which characterizes a positive answer to this question and sheds further light on recent work of VanDieren \cite{vandieren-symmetry-apal, vandieren-chainsat-apal}. The author believes it has independent interest.

Tameness has been used by Grossberg and VanDieren to prove an upward categoricity transfer from categoricity in two successive cardinals \cite{tamenesstwo, tamenessthree}. In Section \ref{categ-sec}, we revisit this result and show that tameness is in some sense needed to prove it. Although this could have been derived from the results of Shelah's books, this seems not to have been noticed before. Nevertheless, the results of this paper show that if an AEC is categorical $\lambda$ and $\lambda^+$ and has a good frame in both $\lambda$ and $\lambda^+$, then it is categorical in $\lambda^{++}$, see Corollary \ref{categ-cor}.

To read this paper, the reader should preferably have a solid knowledge of good frames, including knowing Chapter II of \cite{shelahaecbook}, \cite{jrsh875}, as well as \cite{jarden-tameness-apal}. Still, we have tried to give all the definitions and relevant background facts in Section \ref{prelim-sec}.

The author thanks John T.\ Baldwin, Will Boney, Rami Grossberg, Adi Jarden, Marcos Mazari-Armida, and the referee for comments that helped improve the quality of this paper.

\section{Preliminaries}\label{prelim-sec}

\subsection{Notational conventions}

Given a structure $M$, write write $|M|$ for its universe and $\|M\|$ for the cardinality of its universe. We often do not distinguish between $M$ and $|M|$, writing e.g.\ $a \in M$ or $\ba \in \fct{<\alpha}{M}$ instead of $a \in |M|$ and $\ba \in \fct{<\alpha}{|M|}$. We write $M \subseteq N$ to mean that $M$ is a substructure of $N$.

\subsection{Abstract elementary classes}
An \emph{abstract class} is a pair $\K = (K, \lea)$, where $K$ is a class of structures in a fixed vocabulary $\tau = \tau (\K)$ and $\lea$ is a partial order, $M \lea N$ implies $M \subseteq N$, and both $K$ and $\lea$ respect isomorphisms (the definition is due to Grossberg). Any abstract class admits a notion of \emph{$\K$-embedding}: these are functions $f: M \rightarrow N$ such that $f: M \cong f[M]$ and $f[M] \lea N$.

We often do not distinguish between $K$ and $\K$. For $\lambda$ a cardinal, we will write $\K_\lambda$ for the restriction of $\K$ to models of cardinality $\lambda$. Similarly define $\K_{\ge \lambda}$ or more generally $\K_{S}$, where $S$ is a class of cardinals. We will also use the following notation:

\begin{notation}
  For $\K$ an abstract class and $N \in \K$, write $\PK (N)$ for the set of $M \in \K$ with $M \lea N$. Similarly define $\PKp{\lambda}(N)$, $\PKp{<\lambda}(N)$, etc.
\end{notation}

For an abstract class $\K$, we denote by $\Ii (\K)$ the number of models in $\K$ up to isomorphism (i.e.\ the cardinality of $\K /_{\cong}$). We write $\Ii (\K, \lambda)$ instead of $\Ii (\K_\lambda)$. When $\Ii (\K) = 1$, we say that $\K$ is \emph{categorical}. We say that $\K$ is \emph{categorical in $\lambda$} if $\K_\lambda$ is categorical, i.e.\ $\Ii (\K, \lambda) = 1$.

We say that $\K$ has \emph{amalgamation} if for any $M_0 \lea M_\ell$, $\ell = 1,2$ there is $M_3 \in \K$ and $\K$-embeddings $f_\ell : M_\ell \rightarrow M_3$, $\ell = 1,2$. $\K$ has \emph{joint embedding} if any two models can be $\K$-embedded in a common model. $\K$ has \emph{no maximal models} if for any $M \in \K$ there exists $N \in \K$ with $M \lea N$ and $M \neq N$ (we write $M \lta N$). Localized concepts such as \emph{amalgamation in $\lambda$} mean that $\K_\lambda$ has amalgamation.

The definition of an abstract elementary class is due to Shelah \cite{sh88}:

\begin{defin}\label{aec-def}
  An \emph{abstract elementary class (AEC)} is an abstract class $\K$ in a finitary vocabulary satisfying:

  \begin{enumerate}
  \item Coherence: if $M_0, M_1, M_2 \in \K$, $M_0 \subseteq M_1 \lea M_2$ and $M_0 \lea M_2$, then $M_0 \lea M_1$.
  \item Tarski-Vaught chain axioms: if $\seq{M_i : i \in I}$ is a $\lea$-directed system and $M := \bigcup_{i \in I} M_i$, then:
    \begin{enumerate}
    \item $M \in \K$.
    \item $M_i \lea M$ for all $i \in I$.
    \item Smoothness: if $N \in \K$ is such that $M_i \lea N$ for all $i \in I$, then $M \lea N$.
    \end{enumerate}
  \item Löwenheim-Skolem-Tarski (LST) axiom: there exists a cardinal $\lambda \ge |\tau (\K)| + \aleph_0$ such that for any $N \in \K$ and any $A \subseteq |N|$, there exists $M \in \PKp{\le (|A| + \lambda)} (N)$ with $A \subseteq |M|$. We write $\LS (\K)$ for the least such $\lambda$.
  \end{enumerate}

  We say that an abstract class $\K$ is an \emph{AEC in $\lambda$} if $\K$ satisfies coherence, $\K = \K_\lambda$, $\lambda \ge |\tau (\K)|$, and $\K$ satisfies the Tarski-Vaught chain axioms whenever $|I| \le \lambda$.
\end{defin}
\begin{remark}\label{gen-rmk}
  If $\K$ is an AEC in $\lambda$, then by \cite[II.1.23]{shelahaecbook} there exists a unique AEC $\K^\ast$ such that $\K_\lambda = \K_\lambda^\ast$, $\K_{<\lambda}^\ast = \emptyset$, and $\LS (\K^\ast) = \lambda$. We call $\K^\ast$ the \emph{AEC generated by $\K$.}
\end{remark}

\subsection{Types}

In any abstract class $\K$, we can define a semantic notion of type, called Galois or orbital types in the literature (such types were introduced by Shelah in \cite{sh300-orig}). For $M \in \K$, $A \subseteq |M|$, and $\bb \in \fct{<\infty}{M}$, we write $\gtp_{\K} (\bb / A; M)$ for the orbital type of $\bb$ over $A$ as computed in $M$ (usually $\K$ will be clear from context and we will omit it from the notation). It is the finest notion of type respecting $\K$-embeddings, see \cite[2.16]{sv-infinitary-stability-afml} for a formal definition. When $\K$ is an elementary class, $\gtp (\bb / A; M)$ contains the same information as the usual notion of $\Ll_{\omega, \omega}$-syntactic type, but in general the two notions need not coincide \cite{hs-example}.

The \emph{length} of $\gtp (\bb / A; M)$ is the length of $\bb$. For $M \in \K$ and $\alpha$ a cardinal, we write $\gS_{\K}^\alpha (M) = \gS^\alpha (M)$ for the set of types over $M$ of length $\alpha$. Similarly define $\gS^{<\alpha} (M)$. When $\alpha = 1$, we just write $\gS (M)$. We define naturally what it means for a type to be realized inside a model, to extend another type, and to take the image of a type by a $\K$-embedding.

\subsection{Stability and saturation}

We say that an abstract class $\K$, is \emph{stable in $\lambda$} (for $\lambda$ an infinite cardinal) if $|\gS (M)| \le \lambda$ for any $M \in \K_\lambda$. If $\K$ is an AEC, $\lambda \ge \LS (\K)$, $\K$ is stable in $\lambda$ and $\K$ has amalgamation in $\lambda$, then we will often use without comments the \emph{existence of universal extension} \cite[II.1.16]{shelahaecbook}: for any $M \in \K_\lambda$, there exists $N \in \K_\lambda$ universal over $M$. This means that $M \lea N$ and any extension of $M$ of cardinality $\lambda$ $\K$-embeds into $N$ over $M$.

For $\K$ an AEC and $\lambda > \LS (\K)$, a model $N \in \K$ is called \emph{$\lambda$-saturated} if for any $M \in \K_{<\lambda}$ with $M \lea N$, any $p \in \gS (M)$ is realized in $N$. $N$ is called \emph{saturated} if it is $\|N\|$-saturated.

We will often use without mention the \emph{model-homogeneous = saturated lemma} \cite[II.1.14]{shelahaecbook}: it says that when $\K_{<\lambda}$ has amalgamation, a model $N \in \K$ is $\lambda$-saturated if and only if it is $\lambda$-model-homogeneous. The latter means that for any $M \in \PKp{<\lambda} (N)$, any $M' \in \K_{<\lambda}$ $\lea$-extending $M$ can be $\K$-embedded into $N$ over $M$. In particular, assuming amalgamation and joint embedding, there is at most one saturated model of a given cardinality. We write $\Ksatp{\lambda}$ for the abstract class of $\lambda$-saturated models in $\K$ (ordered by the appropriate restriction of $\lea$).

\subsection{Frames}

Roughly, a \emph{frame} consists of a class of models of the same cardinality together with an abstract notion of nonforking. The idea is that if the frame is ``sufficiently nice'', then it is possible to extend it to cover bigger models as well. This is the approach of Shelah's book \cite{shelahaecbook}, where \emph{good frames} were introduced. There the abstract notion of nonforking is required to satisfy some of the basic properties of forking in a superstable elementary class. We redefine here the definition of a frame (called pre-frame in \cite[III.0.2]{shelahaecbook}, \cite[3.1]{indep-aec-apal}). We give a slightly different definition, as we do not include certain monotonicity axioms as part of the definition. Shelah assumes that nonforking is only defined for a certain class of types he calls the basic types. For generality, we will work in the same setup (except that we define the basic types from the nonforking relation), but the reader will not lose much by pretending that all types are basic (In Shelah's terminology, that the frame is \emph{type-full}). Note that any weakly successful good frame can be extended to a type-full one, see \cite[III.9.6]{shelahaecbook}.

\begin{defin}
  Let $\lambda$ be an infinite cardinal, and $\alpha \le \lambda^+$ be a non-zero cardinal. A \emph{$(<\alpha, \lambda)$-frame} consists of a pair $\s = (\K, \nf)$, where:

  \begin{enumerate}
  \item $\K$ is an abstract class with $\K = \K_\lambda$.
  \item $\nf$ is a $4$-ary relation on pairs $(\ba, M_0, M, N)$, where $M_0 \lea M \lea N$ and $\ba \in \fct{<\alpha}{N}$. We write $\nfs{M_0}{\ba}{M}{N}$ instead of $\nf (\ba, M_0, M, N)$.
  \item $\nf$ respects $\K$-embeddings: if $f: N \rightarrow N'$ is a $\K$-embedding and $M_0 \lea M \lea N$, $\ba \in \fct{<\alpha}{N}$, then $\nfs{M_0}{\ba}{M}{N}$ if and only if $\nfs{f[M_0]}{f (\ba)}{f[M]}{N'}$.
  \end{enumerate}

  We may write $\K_{\s} = (K_{\s}, \leap{\s})$ for $\K$ and $\nf_{\s}$ for $\nf$. We say that $\s$ is \emph{on $\K^\ast$} if $\K_{\s} = \K^\ast$. A $(\le \beta, \lambda)$-frame (for $\beta \le \lambda$) is a $(<\beta^+, \lambda)$-frame. A \emph{$\lambda$-frame} is a $(\le 1, \lambda)$-frame.
\end{defin}
\begin{defin}\label{extend-def}
  We say a $(<\alpha, \lambda)$-frame $\ts$ \emph{extends} a $(<\beta, \lambda)$-frame $\s$ if $\beta \le \alpha$, $\K_{\s} = \K_{\ts}$, and for $M_0 \leap{\s} M \leap{\s} N$ and $\ba \in \fct{<\beta}{N}$, $\nf_{\ts} (M_0, \ba, M, N)$ if and only if $\nf_{\s} (M_0, \ba, M, N)$.
\end{defin}

Since $\nf$ respects $\K$-embeddings, it respects types. Therefore we can define:

\begin{defin}
  Let $\s$ be a $(<\alpha,\lambda)$-frame. For $M_0 \leap{\s} M$ and $p \in \gS^{<\alpha} (M)$, we say that $p$ \emph{does not $\s$-fork over $M_0$} if $\nfs{M_0}{\ba}{M}{N}$ whenever $p = \gtp (\ba / M; N)$. When $\s$ is clear from context, we omit it and just say that $p$ \emph{does not fork over $M_0$}.
\end{defin}

We will call a type \emph{basic} if it interacts with the nonforking relation, i.e.\ if it does not fork over a $\K$-substructure of its base. Note that basic types are respected by $\K$-embeddings (by the definition of a frame).

\begin{defin}
  Let $\s$ be a $(<\alpha, \lambda)$-frame. For $M \in \K_{\s}$, a type $p \in \gS^{<\alpha} (M)$ will be called \emph{$\s$-basic} (or just \emph{basic} when $\s$ is clear from context) if there exists $M_0 \leap{\s} M$ such that $p$ does not $\s$-fork over $M_0$. We say that $\s$ is \emph{type-full} if for any $M \in \K_{\s}$, every type in $\gS^{<\alpha} (M)$ is basic.
\end{defin}

We will often consider frames whose underlying class is categorical:

\begin{defin}\label{categ-frame-def}
  We say that a $(<\alpha, \lambda)$-frame $\s$ is \emph{categorical} if $\K_{\s}$ is categorical (i.e.\ it contains a single model up to isomorphism).
\end{defin}

We will consider the following properties that forking may have in a frame:

\begin{defin}
  Let $\s$ be a $(<\alpha, \lambda)$-frame.

  \begin{enumerate}
  \item $\s$ has \emph{non-order} if whenever $M_0 \leap{\s} M \leap{\s} N$, $\ba, \bb \in \fct{<\alpha}{M}$ and $A := \ran  (\ba) = \ran (\bb)$, then $\nfs{M_0}{\ba}{M}{N}$ if and only if $\nfs{M_0}{\bb}{M}{N}$. In this case we will write $\nfs{M_0}{A}{M}{N}$.
  \item $\s$ has \emph{monotonicity} if whenever $M_0 \leap{\s} M_0' \leap{\s} M \leap{\s} N \leap{\s} N'$, $\ba \in \fct{<\alpha}{N'}$, $I \subseteq \dom (\ba)$, and $\nfs{M_0}{\ba}{N}{N'}$, we have that $\nfs{M_0'}{\ba \rest I}{M}{N'}$.
  \item $\s$ has \emph{density of basic types} if whenever $M \ltap{\s} N$ and $\beta < \alpha$, there exists $\ba \in \fct{\beta}{N \backslash M}$ such that $\gtp (\ba / M; N)$ is basic.
  \item $\s$ has \emph{disjointness} if $\nfs{M_0}{\ba}{M}{N}$ and $\ba \in \fct{<\alpha}{M}$ imply $\ba \in \fct{<\alpha}{M_0}$.
  \item $\s$ has \emph{existence} if whenever $M \leap{\s} N$, any basic $p \in \gS^{<\alpha} (M)$ has a nonforking extension to $\gS^{<\alpha} (N)$.
  \item $\s$ has \emph{uniqueness} if whenever $M \leap{\s} N$, if $p, q \in \gS^{<\alpha} (N)$ both do not fork over $M$ and $p \rest M = q \rest M$, then $p = q$.
  \item $\s$ has \emph{local character} if for any $\beta \le \min (\alpha, \lambda)$, any limit ordinal $\delta < \lambda^+$ with $\cf{\delta} \ge \beta$, any $\leap{\s}$-increasing continuous chain $\seq{M_i : i \le \delta}$, and any basic $p \in \gS^{<\beta} (M_\delta)$, there exists $i < \delta$ such that $p$ does not fork over $M_i$.
  \item $\s$ has \emph{symmetry} if the following are equivalent for any $M \leap{\s} N$, $\ba, \bb \in \fct{<\alpha}{N}$ such that $\gtp (\ba / M; N)$ and $\gtp (\bb / M; N)$ are both basic:
    
    \begin{enumerate}
    \item There exists $M_{\ba}, N_{\ba}$ such that $N \leap{\s} N_{\ba}$, $M \leap{\s} M_{\ba} \leap{\s} N_{\ba}$, $\ba \in \fct{<\alpha}{M_{\ba}}$, and $\gtp (\bb / M_{\ba}; N_{\ba})$ does not fork over $M$.
    \item There exists $M_{\bb}, N_{\bb}$ such that $N \leap{\s} N_{\bb}$, $M \leap{\s} M_{\bb} \leap{\s} N_{\bb}$, $\bb \in \fct{<\alpha}{M_{\bb}}$, and $\gtp (\ba / M_{\bb}; N_{\bb})$ does not fork over $M$.
    \end{enumerate}
  \item When $\alpha = \lambda^+$, $\s$ has \emph{long transitivity} if it has non-order and whenever $\gamma < \lambda^+$ is an ordinal (not necessarily limit), $\seq{M_i : i \le \gamma}$, $\seq{N_i : i \le \gamma}$ are $\leap{\s}$-increasing continuous, and $\nfs{M_i}{N_i}{M_{i + 1}}{N_{i + 1}}$ for all $i < \gamma$, we have that $\nfs{M_0}{N_0}{M_\gamma}{N_\gamma}$.
    \end{enumerate}
\end{defin}

Shelah's definition of a good frame (for types of length one) says that a frame must have all the properties above and its underlying class must be ``reasonable'' \cite[II.2.1]{shelahaecbook}. The prototypical example is the class of models of cardinality $\lambda$ of a superstable elementary class which is stable in $\lambda$. Taking this class with nonforking gives a good $\lambda$-frame (even a good $(\le \lambda, \lambda)$-frame). We use the definition from \cite[2.1.1]{jrsh875} (we omit the continuity property since it follows, see \cite[2.1.4]{jrsh875}). We add the long transitivity property from \cite[II.6.1]{shelahaecbook} when the types have length $\lambda$.

\begin{defin}
  We say a $(<\alpha, \lambda)$-frame $\s$ is \emph{good} if:

  \begin{enumerate}
  \item $\K_{\s}$ is an AEC in $\lambda$ (see the end of Definition \ref{aec-def}), and $\K_\s \neq \emptyset$. Moreover $\K_{\s}$ is stable in $\lambda$, has amalgamation, joint embedding, and no maximal models.
  \item $\s$ has non-order, monotonicity, density of basic types, disjointness, existence, uniqueness, local character, symmetry, and (when $\alpha = \lambda^+$) long transitivity.
  \end{enumerate}
\end{defin}
\begin{remark}
  We also call the AEC generated by $\K_{\s}$ (see Remark \ref{gen-rmk}) the \emph{AEC generated by $\s$}.
\end{remark}

We will use the following conjugation property of good frames at a crucial point in the proof of Theorem \ref{main-thm}:

\begin{fact}[III.1.21 in \cite{shelahaecbook}]\label{conj-fact}
  Let $\s$ be a categorical good $\lambda$-frame (see Definition \ref{categ-frame-def}). Let $M \leap{\s} N$ and let $p \in \gS (N)$. If $p$ does not fork over $M$, then $p$ and $p \rest M$ are conjugate. That is, there exists an isomorphism $f: N \cong M$ such that $f (p) = p \rest M$.
\end{fact}

\begin{remark}
  The results of this paper also carry over (with essentially the same proofs) in the slightly weaker framework of semi-good $\lambda$-frames introduced in \cite{jrsh875}, where only ``almost stability'' (i.e.\ $|\gS (M)| \le \lambda^+$ for all $M \in \K_{\lambda}$) and the conjugation property are assumed. For example, in Corollary \ref{main-cor-gch} we can assume only that there is a semi-good $\lambda$-frame on $\K_{\lambda}$ with conjugation (but we still should assume there is a good $\lambda^+$-frame on $\Ksatp{\lambda^+}_{\lambda^+}$).
\end{remark}

Several sufficient conditions for the existence of a good frame are known. Assuming GCH, Shelah showed that the existence of a good $\lambda$-frame follows from categoricity in $\lambda$, $\lambda^+$, and a medium number of models in $\lambda^{++}$ (see Fact \ref{shelah-categ-frame}). Good frames can also be built using a small amount of tameness  that follows from amalgamation, no maximal models, and categoricity in a sufficiently high cardinal (see Section \ref{tameness-sec} and Fact \ref{tameness-categ}).

Note that on a categorical good $\lambda$-frame, there is only one possible notion of nonforking with the required properties. In fact, nonforking can be given an explicit description, see \cite[\S 9]{indep-aec-apal}. We will use this without comments:

\begin{fact}[Canonicity of categorical good frames]\label{canon-fact}
  If $\s$ and $\ts$ are categorical good $\lambda$-frames with the same basic types and $\K_{\s} = \K_{\ts}$, then $\s = \ts$.
\end{fact}

\subsection{Superlimits}

As has been done in several recent papers (e.g.\ \cite{jrsh875, indep-aec-apal, downward-categ-tame-apal}), we have dropped the requirement that $\K$ has a superlimit in $\lambda$ from Shelah's definition of a good frame:

\begin{defin}[I.3.3 in \cite{shelahaecbook}]\label{sl-def}
  Let $\K$ be an AEC and let $\lambda \ge \LS (\K)$. A model $N$ is \emph{superlimit in $\lambda$} if:

  \begin{enumerate}
  \item $N \in \K_\lambda$ and $N$ has a proper extension.
  \item $N$ is universal: any $M \in \K_\lambda$ $\K$-embeds into $N$.
  \item For any limit ordinal $\delta < \lambda^+$ and any increasing chain $\seq{N_i : i < \delta}$, if $N \cong N_i$ for all $i < \delta$, then $N \cong \bigcup_{i < \delta} N_i$.
  \end{enumerate}

  We say that $M$ is \emph{superlimit} if it is superlimit in $\|M\|$.
\end{defin}

Again, for a prototypical example consider a superstable elementary class which is stable in a cardinal $\lambda$ and let $M$ be a saturated model of cardinality $\lambda$. Then because unions of chains of $\lambda$-saturated are $\lambda$-saturated, $M$ is superlimit in $\lambda$. In fact, an elementary class has a superlimit in some high-enough cardinal if and only if it is superstable \cite[3.1]{sh868}.

There are no known examples of good $\lambda$-frames that do \emph{not} have a superlimit in $\lambda$. In fact most constructions of good $\lambda$-frames give one, see for example \cite[6.4]{vv-symmetry-transfer-afml}. When a good $\lambda$-frame $\s$ has a superlimit, we can restrict $\s$ to the AEC generated by this superlimit (i.e\ the unique AEC $\K^\ast$ such that $\K_\lambda$ consists of isomorphic copies of the superlimit and is ordered by the appropriate restriction of $\lea$) and obtain a new good frame that will be \emph{categorical} in the sense of Definition \ref{categ-frame-def}. Thus in this paper we will often assume that the good frame is categorical to start with.

When one has a good $\lambda$-frame, it is natural to ask whether the frame can be extended to a good $\lambda^+$-frame. It turns out that the behavior of the saturated model in $\lambda^+$ can be crucial for this purpose. Note that amalgamation in $\lambda$ and stability in $\lambda$ indeed imply that there is a unique, universal, saturated model in $\lambda^+$ (amalgamation in $\lambda^+$ is not needed for this purpose). It is also known that there are no maximal models in $\lambda^+$ (see \cite[II.4.13]{shelahaecbook} or \cite[3.1.9]{jrsh875}). A key property is whether union of chains of saturated models of cardinality $\lambda^+$ are saturated. In fact, it is easy to see that this is equivalent to the existence of a superlimit in $\lambda^+$. We will use it without comments and leave the proof to the reader:

\begin{fact}\label{sl-easy}
  Let $\s$ be a good $\lambda$-frame generating an AEC $\K$. The following are equivalent:

  \begin{enumerate}
  \item $\K$ has a superlimit in $\lambda^+$.
  \item The saturated model in $\K$ of cardinality $\lambda^+$ is superlimit.
  \item For any limit $\delta < \lambda^{++}$ and any increasing chain $\seq{M_i : i < \delta}$ of saturated model in $\K_{\lambda^+}$, $\bigcup_{i < \delta} M_i$ is saturated.
  \end{enumerate}
\end{fact}

\subsection{Tameness}\label{tameness-sec}

In an elementary class, types coincides with sets of formulas so are in particular determined by their restrictions to small subsets of their domain. One may be interested in studying AECs where types have a similar behavior. Such a property is called tameness. Tameness was extracted from an argument of Shelah \cite{sh394} and made into a definition by Grossberg and VanDieren \cite{tamenessone} who used it to prove an upward categoricity transfer theorem \cite{tamenesstwo, tamenessthree}. Here we present the definitions we will use, and a few sufficient conditions for tameness (given as motivation but \emph{not} used in this paper). We refer the reader to the survey of Boney and the author \cite{bv-survey-bfo} for more on tame AECs.

\begin{defin}
  For an abstract class $\K$, a class of types $\Gamma$, and an infinite cardinal $\chi$, we say that $\K$ is \emph{$(<\chi)$-tame} for $\Gamma$ if for any $p, q \in \Gamma$ over the same set $B$, $p \rest A = q \rest A$ for all $A \subseteq B$ with $|A| < \chi$ implies that $p = q$. We say that $\K$ is \emph{$\chi$-tame for $\Gamma$} if it is $(<\chi^+)$-tame for $\Gamma$.
\end{defin}

We will use the following variation from \cite[11.6]{baldwinbook09}:

\begin{defin}
  An AEC $\K$ is \emph{$(<\chi, \lambda)$-weakly tame} if $\K$ is $(<\chi)$-tame for the class of types of length one over \emph{saturated} models of cardinality $\lambda$. When we omit the weakly, we mean that $\K$ is $(<\chi)$-tame for the class of types of length one over any model of cardinality $\lambda$. Define similarly variations such as \emph{$(\chi, <\lambda)$-weakly tame}, or \emph{$\chi$-tame} (which means $(\chi, \lambda)$-tame for all $\lambda \ge \chi$).
\end{defin}

A consequence of the compactness theorem is that any elementary class is $(<\aleph_0)$-tame (for any class of types). Boney \cite{tamelc-jsl}, building on work of Makkai and Shelah \cite{makkaishelah} showed that tameness follows from a large cardinal axiom:

\begin{fact}\label{tamelc}
  If $\K$ is an AEC and $\chi > \LS (\K)$ is strongly compact, then $\K$ is $(<\chi)$-tame for any class of types.
\end{fact}

Recent work of Boney and Unger \cite{lc-tame-pams} show that this can, in a sense, be reversed: the statement ``For every AEC $\K$ there is $\chi$ such that $\K$ is $\chi$-tame'' is \emph{equivalent} to a large cardinal axiom. It is however known that when the AEC is stability-theoretically well-behaved, some amount of tameness automatically holds. This was observed for categoricity in a cardinal of high-enough cofinality by Shelah \cite[II.2.3]{sh394} and later improved to categoricity in any big-enough cardinal by the author \cite[5.7(5)]{categ-saturated-afml}:

\begin{fact}\label{tameness-categ}
  Let $\K$ be an AEC with arbitrarily large models and let $\lambda \ge \ehanf{\LS (\K)}$. If $\K_{<\lambda}$ has amalgamation and no maximal models and $\K$ is categorical in $\lambda$, then there exists $\chi < \ehanf{\LS (\K)}$ such that $\K$ is $(\chi, <\lambda)$-weakly tame.
\end{fact}

Very relevant to this paper is the fact that tameness allows one to transfer good frames up \cite[6.9]{tame-frames-revisited-jsl}:

\begin{fact}\label{ext-frame}
  Let $\s$ be a good $\lambda$-frame generating an AEC $\K$. If $\K$ is $(\lambda, \lambda^+)$-tame and has amalgamation in $\lambda^+$, then there is a good $\lambda^+$-frame $\ts$ on $\K_{\lambda^+}$. Moreover $\ts$-nonforking can be described in terms of $\s$-nonforking as follows: for $M \leap{\ts} N$, $p \in \gS (N)$ does not $\ts$-fork over $M$ if and only if there exists $M_0 \in \PKp{\lambda}{M}$ so that for all $N_0 \in \PKp{\lambda}{N}$ with $M_0 \lea N_0$, $p \rest N_0$ does not $\s$-fork over $M_0$.
\end{fact}
\begin{defin}[{\cite[6.3]{tame-frames-revisited-jsl}}]\label{s-lambdap-def}
  Let $\s$ be a good $\lambda$-frame generating an AEC $\K$. We write $\s_{\lambda^+}$ for the frame on $\K_{\lambda^+}$ with nonforking relation described by Fact \ref{ext-frame}. We write $\s \rest \Ksatp{\lambda^+}_{\lambda^+}$ for its restriction to $\Ksatp{\lambda^+}_{\lambda^+}$.
\end{defin}

\section{When is there a superlimit in $\lambda^+$?}\label{decent-sec}

Starting with a good $\lambda$-frame generating an AEC $\K$, it is natural to ask when $\K$ has a superlimit in $\lambda^+$, i.e.\ when the union of any increasing chains of $\lambda^+$-saturated models is $\lambda^+$-saturated. We should say that there are no known examples when this fails, but we are unable to prove it unconditionally. We give here the following condition on forking characterizing the existence of a superlimit in $\lambda^+$. The condition is extracted from the property $(\ast \ast)_{M_1^\ast, M_2^\ast}$ in \cite[II.8.5]{shelahaecbook}. After the initial submission of this article, John T.\ Baldwin pointed out that the condition is similar, in the first-order case, to the fact (valid in any stable theory) that $\lambda^+$-saturated models are \emph{strongly $\lambda^+$-saturated} (in Shelah's terminology, $F_{\lambda^+}^a$-saturated), see \cite{strong-saturation}.

Recall that whenever $\s$ is a good $\lambda$-frame, there is always a unique saturated model in $\K_{\lambda^+}$ (amalgamation in $\lambda^+$ is not needed to build it), so the definition below always makes sense.

\begin{defin}\label{decent-def}
  Let $\s$ be a good $\lambda$-frame generating an AEC $\K$. We say that $\s$ is \emph{decent} if for any two saturated model $M \lea N$ both in $\K_{\lambda^+}$, any two $\lambda$-sized models $M_0 \lea N_0 \lea N$ with $M_0 \lea M$, and any basic $p \in \gS (M_0)$, the nonforking extension of $p$ to $N_0$ is already realized inside $M$. In other words, there exists $a \in M$ such that $\gtp (a / N_0; N)$ is the nonforking extension of $p$ to $N_0$.
\end{defin}

Note that the type $p$ in the above definition will of course be realized inside $N$ (since $N$ is saturated and contains $N_0$). The point is that it will also be realized inside the smaller model $M$, even though $M$ itself will not necessarily contain $N_0$.

\begin{thm}\label{decent-thm}
  Let $\s$ be a good $\lambda$-frame generating an AEC $\K$. The following are equivalent:

  \begin{enumerate}
  \item\label{cond-1} $\K$ has a superlimit model in $\lambda^+$.
  \item\label{cond-2} There exists a partial order $\le^\ast$ on $\Ksatpp{K}{\lambda^+}_{\lambda^+}$ such that:
    \begin{enumerate}
      \item Whenever $M_0 \lea M_1$ are both in $\Ksatp{\lambda^+}_{\lambda^+}$, there exists $M_2 \in \Ksatp{\lambda^+}_{\lambda^+}$ such that $M_1 \lea M_2$ and $M_0 \le^\ast M_2$
      \item For any increasing chain $\seq{M_i : i < \omega}$ in $\Ksatp{\lambda^+}_{\lambda^+}$ such that $M_i \le^\ast M_{i + 1}$ for all $i < \omega$, $\bigcup_{i < \omega} M_i$ is saturated.
    \end{enumerate}
  \item\label{cond-3} $\s$ is decent.
  \end{enumerate}
\end{thm}
\begin{proof} \
  \begin{itemize}
  \item \underline{(\ref{cond-1}) implies (\ref{cond-2})}: Trivial (take $\le^\ast$ to be $\lea$).
  \item \underline{(\ref{cond-3}) implies (\ref{cond-1})}: Assume (\ref{cond-3}). Let $\delta < \lambda^{++}$ be a limit ordinal and let $\seq{M_i : i < \delta}$ be an increasing chain of saturated models in $\K_{\lambda^+}$. We want to show that $M_\delta := \bigcup_{i < \delta} M_i$ is saturated. Without loss of generality, $\delta = \cf{\delta} < \lambda^+$. Let $M^0 \in \PKp{\lambda}(M_\delta)$. Build $\seq{M_i^0 : i < \delta}$ increasing chain in $\K_\lambda$ such that for all $i < \delta$, $M_i^0 \lea M_i$ and $|M^0| \cap |M_i| \subseteq |M_{i + 1}^0|$. This is possible using that each $M_i$ is saturated. Let $M_\delta^0 := \bigcup_{i < \delta} M_i^0$. Then $M^0 \lea M_\delta^0$. Let $p \in \gS (M_\delta^0)$. By local character, there exists $i < \delta$ such that $p$ does not fork over $M_i^0$. Since $\s$ is decent, $p$ is realized in $M_i$, hence in $M_\delta$. This shows that $M_\delta$ realizes all basic types over its $\K$-substructures of cardinality $\lambda$, which in turns implies that $M_\delta$ is saturated \cite[II.4.3]{shelahaecbook}.

  \item \underline{(\ref{cond-2}) implies (\ref{cond-3})}: Assume (\ref{cond-2}) and suppose for a contradiction that $\s$ is not decent. We build $\seq{M_i : i \le \omega}$, $\seq{M_i^0 : i \le \omega}$ increasing continuous and a type $p \in \gS (M_0^0)$ such that for all $i < \omega$: (writing $p_{M}$ for the nonforking extension of $p$ to $\gS (M)$):

  \begin{enumerate}
  \item $M_i \in \Ksatp{\lambda^+}_{\lambda^+}$.
  \item $M_{i} \le^\ast M_{i + 1}$.
  \item $M_i^0 \in \PKp{\lambda}(M_i)$.
  \item $p_{M_{i + 1}^0}$ is \emph{not} realized in $M_i$.
  \end{enumerate}

  \underline{This is enough}: By assumption, $M_{\omega}$ is saturated. Therefore $p_{M_{\omega}^0}$ is realized inside $M_{\omega}$, and therefore inside $M_i$ for some $i < \omega$. This means in particular that $p_{M_{i + 1}}$ is realized in $M_i$, a contradiction.
  
  \underline{This is possible}: For $i = 0$, pick $M_0, M_1', M_0^0, M_1^0, p \in \gS (M_0^0)$ witnessing the failure of being decent (i.e.\ $M_0 \lea M_1'$ are both saturated in $\K_{\lambda^+}$, $M_0^0 \lea M_0$, $M_0^0 \lea M_1^0 \lea M_1'$, $M_0^0, M_1^0 \in \K_\lambda$, and $p_{M_1^0}$ is omitted in $M_)$), and then use the properties of $\le^\ast$ to obtain $M_1 \in \Ksatp{\lambda^+}_{\lambda^+}$ such that $M_0 \le^\ast M_1$ and $M_1' \lea M_1$. Now given $\seq{M_j^0 : j \le i + 1}$, $\seq{M_j : j \le i + 1}$, $M_i$ and $M_{i + 1}$ are both saturated and so must be isomorphic over any common submodel of cardinality $\lambda$. Let $f: M_i \cong_{M_i^0} M_{i + 1}$ and let $g : M_{i + 1} \cong M_{i + 2}'$ be an extension of $f$ (in particular $M_{i + 1} \lea M_{i + 2}'$). Since $g$ is an isomorphism, $g (p_{M_{i + 1}^0})$ is not realized in $g[M_i] = M_{i + 1}$. Pick $M_{i + 2}^0$ in $\PKp{\lambda} (M_{i + 2}')$ such that $M_{i + 1}^0 \lea M_{i + 2}^0$ and $g[M_{i + 1}^0] \lea M_{i + 2}^0$. Finally, pick some $M_{i + 2} \in \Ksatp{\lambda^+}_{\lambda^+}$ such that $M_{i + 1} \le^\ast M_{i + 2}$ and $M_{i + 2}' \lea M_{i + 2}$. This is possible by the assumed properties of $\le^\ast$.
  \end{itemize}
\end{proof}
\begin{remark}
  Theorem \ref{decent-thm} can be generalized to the weaker framework of a $\lambda$-superstable AEC (implicit for example in \cite{gvv-mlq}; see \cite{uq-forking-mlq} for what forking is there and what properties it has). For this we ask in the definition of decent that $M$ be limit, and we ask for example that $M_{i + 1}$ is limit over $M_i$ in the proof of (\ref{cond-3}) implies (\ref{cond-1}) of Theorem \ref{decent-thm}. VanDieren \cite{vandieren-symmetry-apal} has shown that $\lambda$-symmetry (a property akin to the symmetry property of good $\lambda$-frame, see \cite[2.18]{uq-forking-mlq}) is equivalent to the property that reduced towers are continuous, and if there is a superlimit in $\lambda^+$, then reduced towers are continuous. Thus decent is to ``superlimit in $\lambda^+$'' what $\lambda$-symmetry is to ``reduced towers are continuous''. In particular, being decent is a strengthening of the symmetry property. Note also that taking $\le^\ast$ in (\ref{cond-2}) of Theorem \ref{decent-thm} as ``being universal over'', we obtain an alternate proof of the main theorem of \cite{vandieren-chainsat-apal} which showed that $\lambda$ and $\lambda^+$-superstability together with the uniqueness of limit models in $\lambda^+$ imply that the union of a chain of $\lambda^+$-saturated models is $\lambda^+$-saturated (see also the proof of Lemma \ref{decent-from-frame} here).
\end{remark}

A property related to being decent is what Shelah calls $\goodp$. We show that $\goodp$ implies decent. We do not know whether the converse holds but it seems (see Fact \ref{reflect-equiv-fact}) that wherever Shelah uses $\goodp$, he only needs decent.

\begin{defin}[III.1.3 in \cite{shelahaecbook}]\label{goodp-def}
  Let $\s$ be a good $\lambda$-frame generating an AEC $\K$. We say that $\s$ is \emph{$\goodp$} if the following is \emph{impossible}:
    
    There exists increasing continuous chains in $\K_\lambda$ $\seq{M_i : i < \lambda^+}$, $\seq{N_i : i < \lambda^+}$, a basic type $p \in \gS (M_0)$, and $\seq{a_i : i < \lambda^+}$ such that for any $i < \lambda^+$:

  \begin{enumerate}
  \item $M_i \lea N_i$.
  \item $a_{i + 1} \in |M_{i + 2}|$ and $\gtp (a_{i + 1}, M_{i + 1}, M_{i + 2})$ is a nonforking extension of $p$, but $\gtp (a_{i + 1}, N_0, N_{i + 2})$ forks over $M_0$.
  \item $\bigcup_{j < \lambda^+} M_j$ is saturated.
  \end{enumerate}
\end{defin}

\begin{lem}\label{goodp-decent}
  Let $\s$ be a good $\lambda$-frame. If $\s$ is $\goodp$, then $\s$ is decent.
\end{lem}
\begin{proof}
  Let $\K$ be the AEC generated by $\s$. Suppose that $\s$ is not decent. Fix witnesses $M, N, M^0, N^0, p$. We build increasing continuous chains in $\K_\lambda$, $\seq{M_i : i < \lambda^+}$, $\seq{N_i : i < \lambda^+}$ and $\seq{a_i :i < \lambda^+}$ such that for all $i < \lambda^+$:

  \begin{enumerate}
  \item $M_0 = M^0$, $N_0 = N^0$.
  \item $M_i \lea M$, $N_i \lea N$.
  \item $M_i \lea N_i$.
  \item $M_{i + 1}$ is universal over $M_i$.
  \item $a_i \in M_{i + 1}$.
  \item $\gtp (a_i / M_i; M_{i + 1})$ is a nonforking extension of $p$.
  \end{enumerate}

  This is possible since both $M$ and $N$ are saturated. This is enough: $M_{\lambda^+} := \bigcup_{i < \lambda^+} M_i$ is saturated (we could also require that $M_{\lambda^+} = M$ but this is not needed) and for any $i < \lambda^+$, $\gtp (a_{i + 1} / N_0; N_{i + 2})$ forks over $M_0$. If not, then by the uniqueness property of nonforking, we would have that $a_{i + 1}$ realizes the nonforking extension of $p$ to $N_0 = N^0$, which we assumed was impossible.
\end{proof}

\section{From weak tameness to good frame}

In this section, we briefly investigate how to generalize Fact \ref{ext-frame} to AECs that are only $(\lambda, \lambda^+)$-weakly tame. The main problem is that the class $\Ksatp{\lambda^+}_{\lambda^+}$ may not be closed under unions of chains (i.e.\ it may not be an AEC in $\lambda^+$). Indeed, this is the only difficulty. 

\begin{thm}\label{frame-ext-weak}
  Let $\s$ be a good $\lambda$-frame generating an AEC $\K$. If $\K$ is $(\lambda, \lambda^+)$-weakly tame, the following are equivalent:

  \begin{enumerate}
  \item\label{weak-ext-1} $\s$ is decent and $\Ksatp{\lambda^+}_{\lambda^+}$ has amalgamation.
  \item\label{weak-ext-2} $\s_{\lambda^+} \rest \Ksatp{\lambda^+}_{\lambda^+}$ (see Definition \ref{s-lambdap-def}) is a good $\lambda^+$-frame on $\Ksatp{\lambda^+}_{\lambda^+}$.
  \item\label{weak-ext-3} There is a good $\lambda^+$-frame on $\Ksatp{\lambda^+}_{\lambda^+}$.
  \end{enumerate}

  Moreover in (\ref{weak-ext-1}), if $\s$ is type-full, then we can conclude that the good $\lambda^+$-frame described in (\ref{weak-ext-2}) is also type-full.
\end{thm}
\begin{proof}
  (\ref{weak-ext-1}) implies (\ref{weak-ext-2}) is exactly as in the proof of Fact \ref{ext-frame}. (\ref{weak-ext-2}) implies (\ref{weak-ext-3}) is trivial. (\ref{weak-ext-3}) implies (\ref{weak-ext-1}) is by Fact \ref{sl-easy} and Theorem \ref{decent-thm}, since by definition the existence of a good $\lambda^+$-frame on $\Ksatp{\lambda^+}_{\lambda^+}$ implies that $\Ksatp{\lambda^+}_{\lambda^+}$ has amalgamation and $\Ksatp{\lambda^+}_{\lambda^+}$ is an AEC in $\lambda^+$, hence that the union of an increasing chain of $\lambda^+$-saturated models is $\lambda^+$-saturated, i.e.\ that $\K$ has a superlimit in $\lambda^+$.
\end{proof}
\begin{remark}
  A note on the definitions: statements of the form ``there is a good $\lambda^+$-frame on $\Ksatp{\lambda^+}_{\lambda^+}$'' imply in particular (from the definition of a good frame) that $\Ksatp{\lambda^+}_{\lambda^+}$ is an AEC in $\lambda^+$, i.e.\ that the union of an increasing chain of $\lambda^+$-saturated models is $\lambda^+$-saturated. We will often use this without comments.
\end{remark}

It is natural to ask whether the good $\lambda^+$-frame in Theorem \ref{frame-ext-weak} is itself decent. We do not know the answer, but can answer positively assuming $2^{\lambda} < 2^{\lambda^+}$ (see Corollary \ref{main-cor-gch}) and in fact in this case $\s$ is also $\goodp$. We can show in ZFC that being $\goodp$ transfers up. The proof adapts an argument of Shelah \cite[III.1.6(2)]{shelahaecbook}.

\begin{thm}
  Let $\s$ be a $\goodp$ $\lambda$-frame generating an AEC $\K$. If $\K$ is $(\lambda, \lambda^+)$-weakly tame and $\Ksatp{\lambda^+}_{\lambda^+}$ has amalgamation, then $\s_{\lambda^+} \rest \Ksatp{\lambda^+}_{\lambda^+}$ is a $\goodp$ $\lambda^+$-frame on $\Ksatp{\lambda^+}_{\lambda^+}$.
\end{thm}
\begin{proof}
  By Lemma \ref{goodp-decent}, $\s$ is decent, so by Theorem \ref{frame-ext-weak} $\ts := \s_{\lambda^+} \rest \Ksatp{\lambda^+}_{\lambda^+}$ is a good $\lambda^+$-frame on $\Ksatp{\lambda^+}_{\lambda^+}$. Moreover by definition $\ts$-nonforking is described in terms of $\s$-nonforking as in the statement of Fact \ref{ext-frame}.

  Suppose that $\ts$ is \emph{not} $\goodp$. Let $\seq{M_i : i < \lambda^{++}}$, $\seq{N_i : i < \lambda^{++}}$, $p$, $\seq{a_i : i < \lambda^{++}}$ witness that $\ts$ is not $\goodp$. Let $\seq{M_j' : j < \lambda^+}$ be increasing continuous in $\K_\lambda$ such that $M_0 = \bigcup_{j < \lambda^+} M_j'$ and let $\seq{N_j' : j < \lambda^+}$ be increasing continuous in $\K_\lambda$ such that $N_0 = \bigcup_{j < \lambda^+} N_j'$ and $M_j' \lea N_j'$ for all $j < \lambda^+$.

  By a standard pruning argument, there is $j^\ast < \lambda^+$ and an unbounded $S \subseteq \lambda^{++}$ of successor ordinals such that for all $i \in S$ and all $M' \in \PKp{\lambda} (M_i)$ with $M_{j^\ast}' \lea M'$, $\gtp (a_i / M'; M_{i + 1})$ does not $\s$-fork over $M_{j^\ast}'$. Now by assumption for all $i \in S$, $\gtp (a_i / N_0; N_{i + 1})$ $\ts$-forks over $M_0$, so by a pruning argument again, there is an unbounded $S' \subseteq S$ and $j^{\ast \ast} \in [j^\ast, \lambda^+)$ such that for all $i \in S'$, $\gtp (a_i / N_{j^{\ast \ast}}'; N_{i + 1})$ $\s$-forks over $M_{j^\ast}'$. We build $\seq{i_j : j < \lambda^+}$ and $\seq{M_j^\ast : j < \lambda^+}$, $\seq{N_j^\ast : j < \lambda^+}$ increasing continuous in $\K_\lambda$ such that for all $j < \lambda^+$:

    \begin{enumerate}
    \item $i_j \in S'$.
    \item $M_j^\ast \lea N_j^\ast$.
    \item $M_{j + 1}^\ast$ is universal over $M_j^\ast$.
    \item $M_0^\ast = M_{j^\ast}'$, $N_0^\ast = N_{j^{\ast \ast}}'$.
    \item $M_j^\ast \lea M_{i_j}$, $N_j^\ast \lea N_{i_j}$.
    \item $a_{i_j} \in M_{j + 1}^\ast$.
    \end{enumerate}

    \underline{This is enough}: Then by construction of $S'$, $j^\ast$, and $j^{\ast \ast}$, $\seq{M_j^\ast : j < \lambda^+}$, $\seq{N_j^\ast : j < \lambda^+}$, $\gtp (a_{i_0} / M_0^\ast; M_1^\ast) $ and $\seq{a_{i_j} : j < \lambda^+}$ witness that $\s$ is not $\goodp$.

    \underline{This is possible}: Let $M_{\lambda^{++}} := \bigcup_{i < \lambda^{++}} M_i$, $N_{\lambda^{++}} := \bigcup_{i < \lambda^{++}} N_i$. We are already given $M_0^\ast$ and $N_0^\ast$ and for $j$ limit we take unions. Now assume inductively that $\seq{M_k^\ast : k \le j}$, $\seq{N_k^\ast : k \le j}$ and $\seq{i_k : k < j}$ are already given, with $M_j^\ast \lea M_{\lambda^{++}}$ and $N_j^\ast \lea N_{\lambda^{++}}$. Let $i_j \in S'$ be big-enough such that $N_{i_j}$ contains $N_j^\ast$, $M_{i_j}$ contains $M_{j}^\ast$, and $i_k < i_j$ for all $k < j$. Such an $i_j$ exists since $S'$ is unbounded. Now let $M^\ast \in \PKp{\lambda^+} M_{\lambda^{++}}$ contain $M_j^\ast$ and $a_{i_j}$ and be saturated. Such an $M^\ast$ exists since $M_{\lambda^{++}}$ is saturated by assumption. Now pick $M_{j + 1}^\ast \in \PKp{\lambda} (M^\ast)$ so that $a_{i_j} \in M_{j + 1}^\ast$ and $M_{j + 1}^\ast$ is universal over $M_j^\ast$. This is possible since $M^\ast$ is saturated. Finally, pick any $N_{j + 1}^\ast \in \PKp{\lambda} (N_{\lambda^{++}}) $ containing $M_{j + 1}^\ast$ and $N_j^\ast$.
\end{proof}

We end this section by noting that in the context of Fact \ref{ext-frame}, $\s$ is decent and hence by Theorem \ref{frame-ext-weak} there is a good $\lambda^+$-frame on $\Ksatp{\lambda^+}_{\lambda^+}$. In fact:

\begin{lem}\label{decent-from-frame}
  Let $\s$ be a good $\lambda$-frame generating an AEC $\K$. If there is a good $\lambda^+$-frame on $\K_{\lambda^+}$, then $\s$ is decent.
\end{lem}

The proof uses the uniqueness of limit models in good frames, due to Shelah \cite[II.4.8]{shelahaecbook} (see \cite[9.2]{ext-frame-jml} for a proof): 

\begin{fact}\label{uq-lim}
  Let $\s$ be a good $\lambda$-frame, $\delta_1, \delta_2 < \lambda^+$ be limit ordinals. Let $\seq{M_i^\ell : i \le \delta_\ell}$, $\ell = 1,2$, be increasing continuous. If for all $\ell = 1,2$, $i < \delta_\ell$, $M_{i + 1}^\ell$ is universal over $M_i$, then $M_{\delta_1}^1 \cong M_{\delta_2}^2$.
\end{fact}

\begin{proof}[Proof of Lemma \ref{decent-from-frame}]
  By Theorem \ref{decent-thm}, it suffices to show that $\K$ has a superlimit in $\lambda^+$. We could apply two results of VanDieren \cite{vandieren-symmetry-apal, vandieren-chainsat-apal} but we prefer to give a more explicit proof here.

  Let $\seq{N_i : i \le \lambda^+}$ be increasing continuous in $\K_{\lambda^+}$ such that $N_{i + 1}$ is universal over $N_i$ for all $i < \lambda^+$. This is possible since by definition of a good $\lambda^+$-frame, $\K$ is stable in $\lambda^+$ and has amalgamation in $\lambda^+$. Clearly, $N_{\lambda^+}$ is saturated. Moreover by Fact \ref{uq-lim}, $N_{\lambda^+} \cong N_{\omega}$. Thus $N_\omega$ is also saturated. We chose $\seq{N_i : i \le \omega}$ arbitrarily, therefore (\ref{cond-2}) of Theorem \ref{decent-thm} holds with $\le^\ast$ being ``universal over or equal to''. Thus (\ref{cond-3}) there holds: $\s$ is decent, as desired.
\end{proof}

There are other variations on Lemma \ref{decent-from-frame}. For example, if $\s$ is a good $\lambda$-frame generating an AEC $\K$, $\K$ has amalgamation in $\lambda^+$, and $\K$ is $(\lambda, \lambda^+)$-weakly tame, then $\s$ is decent (to prove this, we transfer enough of the good $\lambda$-frame up, then apply results of VanDieren \cite{vandieren-symmetry-apal, vandieren-chainsat-apal}). 

\section{From good frame to weak tameness}

In this section, we look at a sufficient condition (due to Shelah) implying that a good $\lambda$-frame can be extended to a good $\lambda^+$-frame and prove its necessity.

It turns out it is convenient to first extend the good $\lambda$-frame to a good $(\le \lambda, \lambda)$-frame. For this, the next technical property is of great importance, and it is key in Chapter II and III of \cite{shelahaecbook}. The definition below follows \cite[4.1.5]{jrsh875} (but as usual, we work only with type-full frames). Note that we will not use the exact content of the definition, only its consequence. We give the definition only for the benefit of the curious reader. 

\begin{defin}\label{weakly-succ-def}
  Let $\K$ be an abstract class and $\lambda$ be a cardinal.
  
  \begin{enumerate}
  \item\index{amalgam} For $M_0 \lea M_\ell$ all in $\K_\lambda$, $\ell = 1,2$, an \emph{amalgam of $M_1$ and $M_2$ over $M_0$} is a triple $(f_1, f_2, N)$ such that $N \in \K_\lambda$ and $f_\ell : M_\ell \xrightarrow[M_0]{} N$.
  \item\index{equivalence of amalgam} Let $(f_1^x, f_2^x, N^x)$, $x = a,b$ be amalgams of $M_1$ and $M_2$ over $M_0$. We say $(f_1^a, f_2^a, N^a)$ and $(f_1^b, f_2^b, N^b)$ are \emph{equivalent over $M_0$} if there exists $N_\ast \in \K_\lambda$ and $f^x : N^x \rightarrow N_\ast$ such that $f^b \circ f_1^b = f^a \circ f_1^a$ and $f^b \circ f_2^b = f^a \circ f_2^a$, namely, the following commutes:

  \[
  \xymatrix{ & N^a \ar@{.>}[r]^{f^a} & N_\ast \\
    M_1 \ar[ru]^{f_1^a} \ar[rr]|>>>>>{f_1^b} & & N^b \ar@{.>}[u]_{f^b} \\
    M_0 \ar[u] \ar[r] & M_2 \ar[uu]|>>>>>{f_2^a}  \ar[ur]_{f_2^b} & \\
  }
  \]

  Note that being ``equivalent over $M_0$'' is an equivalence relation (\cite[4.3]{jrsh875}).
\item Let $\s$ be a good $(<\alpha, \lambda)$-frame.
  \begin{enumerate}
    \item A \emph{uniqueness triple} in $\s$ is a triple $(\ba, M, N)$ such that $M \leap{\s} N$, $\ba \in \fct{<\alpha}{N}$ and for any $M_1 \geap{\s} M$, there exists a \emph{unique} (up to equivalence over $M$) amalgam $(f_1, f_2, N_1)$ of $N$ and $M_1$ over $M$ such that $\gtp (f_1 (a) / f_2[M_1] ; N_1)$ does not fork over $M$.
    \item $\s$ has the \emph{existence property for uniqueness triples} if for any $M \in \K_{\s}$ and any basic $p \in \gS^{<\alpha} (M)$, one can write $p = \gtp (\ba / M; N)$ with $(\ba, M, N)$ a uniqueness triple.
    \item We say that $\s$ is \emph{weakly successful} if its restriction to types of length one has the existence property for uniqueness triples.
  \end{enumerate}
  \end{enumerate}
\end{defin}

As an additional motivation, we mention the closely related notion of a \emph{domination triple}:

\begin{defin}
  Let $\s$ be a good $(\le \lambda, \lambda)$-frame. A \emph{domination triple} in $\s$ is a triple $(\ba, M, N)$ such that $M \leap{\s} N$, $\ba \in \fct{\le \lambda}{N}$, and whenever $M' \leap{\s} N'$ are such $M \lea M'$, $N \lea N'$, then $\nfs{M}{\ba}{M'}{N'}$ implies $\nfs{M}{N}{M'}{N'}$.
\end{defin}

The following fact shows that domination triples are the same as uniqueness triples once we have managed to get a type-full good $(\le \lambda, \lambda)$-frame. The advantage of uniqueness triples is that they can be defined already inside a good $\lambda$-frame.

\begin{fact}[11.7, 11.8 in \cite{indep-aec-apal}]\label{domin-coincide}
  Let $\s$ be a type-full good $(\le \lambda, \lambda)$-frame. Then in $\s$ uniqueness triples and domination triples coincide.
\end{fact}

The importance of weakly successful good frames is that they extend to longer types. This is due to Shelah:

\begin{fact}\label{weakly-succ-extension}
  Let $\s$ be a categorical good $\lambda$-frame. If $\s$ is weakly successful, then there is a unique type-full good $(\le \lambda, \lambda)$-frame extending $\s$.
\end{fact}
\begin{proof}
  The uniqueness is \cite[II.6.3]{shelahaecbook}. Existence is the main result of \cite[\S II.6]{shelahaecbook}, although there Shelah only builds a nonforking relation for models satisfying the axioms of a good $(\le \lambda, \lambda)$-frame. How to extend this to all types of length at most $\lambda$ is done in \cite[\S III.9]{shelahaecbook}. An outline of the full argument is in the proof of \cite[12.16(1)]{indep-aec-apal}.
\end{proof}

It is not clear whether the converse of Fact \ref{weakly-succ-extension} holds, but see Fact \ref{reflect-equiv-fact}.

Shelah proved \cite[II.5.11]{shelahaecbook} that a categorical good $\lambda$-frame (generating an AEC $\K$) is weakly successful whenever $2^{\lambda} < 2^{\lambda^+} < 2^{\lambda^{++}}$ and $\K$ has a ``medium'' number of models in $\lambda^{++}$. Shelah has also shown that being weakly successful follows from some stability in $\lambda^+$ and $2^{\lambda} < 2^{\lambda^+}$, see \cite[E.8]{downward-categ-tame-apal} for a proof:

\begin{fact}\label{weakly-succ-gch}
  Let $\s$ be a categorical good $\lambda$-frame generating an AEC $\K$. Assume $2^{\lambda} < 2^{\lambda^+}$. If for every saturated $M \in \K_{\lambda^+}$ there is $N \in \K_{\lambda^+}$ universal over $M$, then $\s$ is weakly successful.
\end{fact}

Once we have a good $(\le \lambda, \lambda)$-frame, we can define a candidate for a good $\lambda^+$-frame on the saturated models in $\K_{\lambda^+}$. Nonforking in $\lambda^+$ is defined in terms of nonforking in $\lambda$ (in fact one can make sense of this definition even if we only start with a good $\lambda$-frame), but the problem is that we do not know that the class of saturated models in $\K_{\lambda^+}$ has amalgamation. To achieve this, the ordering $\lea$ is changed to a new ordering $\leap{\s^+}$ so that nonforking ``reflects down''. The definition is due to Shelah \cite[III.1.7]{shelahaecbook} but we follow \cite[10.1.6]{jrsh875}:

\begin{defin}\label{succ-def}
  Let $\s = (\K_{\s}, \nf)$ be a type-full good $(\le \lambda, \lambda)$-frame generating an AEC $\K$. We define a pair $\s^+ = (\K_{\s^+}, \nf_{\s^+})$ as follows:

  \begin{enumerate}
  \item $\K_{\s^+} = (K_{\s^+}, \leap{\s^+})$, where:
    \begin{enumerate}
    \item $K_{\s^+}$ is the class of saturated models in $\K_{\lambda^+}$.
    \item For $M, N \in K_{\s^+}$, $M \leap{\s^+} N$ holds if and only if there exists increasing continuous chains $\seq{M_i : i < \lambda^+}$, $\seq{N_i : i < \lambda^+}$ in $\K_\lambda$ such that:

  \begin{enumerate}
  \item $M = \bigcup_{i < \lambda^+} M_i$.
  \item $N = \bigcup_{i < \lambda^+} N_i$.
  \item For all $i < j < \lambda^+$, $\nfs{M_i}{N_i}{M_{j}}{N_{j}}$.
  \end{enumerate}
    \end{enumerate}
  \item For $M_0 \leap{\s^+} M \leap{\s^+} N$ and $a \in N$, $\nf_{\s^+} (M_0, a, M, N)$ holds if and only if there exists $M_0' \in \PKp{\lambda} (M_0)$ such that for all $M' \in \PKp{\lambda} (M)$ and all $N' \in \PKp{\lambda} (N)$, if $M_0' \lea M' \lea N'$ and $a \in N'$, then $\nf_{\s} (M_0', a, M', N')$.
  \end{enumerate}

  For a weakly successful categorical good $\lambda$-frame $\s$, we define $\s^+ := \ts^+$, where $\ts$ is the unique extension of $\s$ to a type-full good $(\le \lambda, \lambda)$-frame (Fact \ref{weakly-succ-extension}).
\end{defin}

As a motivation for the definition of $\leap{\s^+}$, observe that if $\seq{M_i : i < \lambda^+}$, $\seq{N_i : i < \lambda^+}$ are increasing continuous in $\K_\lambda$ such that $M_i \lea N_i$ for all $i < \lambda^+$, then it is known that there is a club $C \subseteq \lambda^+$ such that $M_j \cap N_i = M_i$ for all $i \in C$ and all $j > i$. We would like to conclude the stronger property that $\nfs{M_i}{N_i}{M_j}{N_j}$ for $i \in C$ and $j > i$. If we are working in a superstable elementary class this is true as nonforking has a strong finite character property, but in the more general setup of good frames this is not clear. Thus the property is built into the definition by changing the ordering. This creates a new problem: we do not know whether $\K_{\s^+}$ is an AEC (smoothness is the problematic axiom). Note that there are no known examples of weakly successful good $\lambda$-frame $\s$ where $\leap{\s^+}$ is not $\lea$.

The following general properties of $\s^+$ are known. They are all proven in \cite[\S II.7]{shelahaecbook} but we cite from \cite{jrsh875} since the proofs there are more detailed:

\begin{fact}\label{weakly-succ-facts}
  Let $\s$ be a weakly successful categorical good $\lambda$-frame generating an AEC $\K$.

  \begin{enumerate}
  \item\cite[6.1.5]{jrsh875} If $M, N \in \K_{\s^+}$ and $M \leap{\s^+} N$, then $M \lea N$.
  \item\label{weakly-succ-ac} \cite[6.1.6(b),(d), 7.1.18(a)]{jrsh875}, $\K_{\s^+}$ is an abstract class which is closed under unions of chains of length strictly less than $\lambda^{++}$.
  \item\label{weakly-succ-monot} \cite[6.1.6(c)]{jrsh875} $\K_{\s^+}$ satisfies the following strengthening of the coherence axiom: if $M_0, M_1, M_2 \in \K_{\s^+}$ are such that $M_0 \leap{\s^+} M_2$ and $M_0 \lea M_1 \lea M_2$, then $M_0 \leap{\s^+} M_1$.
  \item\cite[6.1.6(e), 7.1.18(c)]{jrsh875} $\K_{\s^+}$ has no maximal models and amalgamation.
  \item\label{weakly-succ-type-eq} \cite[10.6]{jarden-tameness-apal} Let $M \leap{\s^+} N_\ell$ and let $a_\ell \in N_\ell$, $\ell = 1,2$. Then $\gtpp{\K} (a_1 / M; N_1) = \gtpp{\K} (a_2 / M; N_2)$ if and only if $\gtpp{\K_{\s^+}} (a_1 / M; N_1) = \gtpp{\K_{\s^+}} (a_2 / M; N_2)$. In particular, $\s^+$ is a $\lambda^+$-frame.
  \end{enumerate}
\end{fact}

We will use the following terminology, taken from \cite[3.7(2)]{downward-categ-tame-apal}:

\begin{defin}\label{reflect-def}
  Let $\s$ be a type-full good $(\le \lambda, \lambda)$-frame generating an AEC $\K$. We say that $\s$ \emph{reflects down} if $M \lea N$ implies $M \leap{\s^+} N$ for all $M, N \in \K_{\s^+}$. We say that $\s$ \emph{almost reflects down} if $\K_{\s^+}$ is an AEC in $\lambda^+$ (see the bottom of Definition \ref{aec-def}). We say that a weakly successful good $\lambda$-frame \emph{[almost] reflects down} if its extension to a type-full good $(\le \lambda, \lambda)$-frame [almost] reflects down, see Fact \ref{weakly-succ-extension}.
\end{defin}

Shelah uses the less descriptive ``successful'':

\begin{defin}[III.1.1 in \cite{shelahaecbook}]
  We say that a good $\lambda$-frame $\s$ is \emph{successful} if it is weakly successful and almost reflects down.
\end{defin}

The point of this definition is that if it holds, then $\s$ can be extended to a good $\lambda^+$-frame. Moreover $\s$ is successful when there are few models in $\lambda^{++}$:

\begin{fact}[III.1.9 in \cite{shelahaecbook}]\label{good-ext-thm}
  If $\s$ is a successful categorical good $\lambda$-frame, then $\s^+$ is a $\goodp$ $\lambda^+$-frame.
\end{fact}

\begin{fact}[II.8.4, II.8.5 in \cite{shelahaecbook}, or see 7.1.3 in \cite{jrsh875}]\label{succ-nm}
  Let $\s$ be a weakly successful categorical good $\lambda$-frame generating the AEC $\K$. If $\Ii (\K, \lambda^{++}) < 2^{\lambda^{++}}$, then $\s$ is successful.
\end{fact}

The downside of working only with a successful good $\lambda$-frame is that the ordering of $\K_{\s^+}$ may not be $\lea$ anymore. Thus the AEC generated by $\K_{\s^+}$ may be different from the one generated by $\K_{\s}$. For example it may fail to have arbitrarily large models even if the original one does. This is why we will focus on good frames that \emph{reflect down}, i.e.\ $\leap{\s^+}$ is just $\lea$. Several characterizations of this situation are known. (\ref{reflect-equiv-4}) implies (\ref{reflect-equiv-1}) below is due to Jarden and all the other implications are due to Shelah (but as usual we mostly quote from \cite{jrsh875}):

\begin{fact}\label{reflect-equiv-fact}
  Let $\s$ be a categorical good $\lambda$-frame. The following are equivalent:

  \begin{enumerate}
  \item\label{reflect-equiv-1} $\s$ extends to a type-full good $(\le \lambda, \lambda)$-frame which reflects down.
  \item\label{reflect-equiv-2} $\s$ is successful and $\goodp$ (see Definition \ref{goodp-def}).
  \item\label{reflect-equiv-3} $\s$ is successful and decent (see Definition \ref{decent-def}).
  \item\label{reflect-equiv-4} $\s$ is weakly successful and generates an AEC $\K$ which is $(\lambda, \lambda^+)$-weakly tame and so that $\Ksatp{\lambda^+}_{\lambda^+}$ has amalgamation.
  \end{enumerate}
\end{fact}
\begin{proof} \
  \begin{itemize}
  \item \underline{(\ref{reflect-equiv-1}) implies (\ref{reflect-equiv-2})}: By \cite[3.11]{downward-categ-tame-apal}, $\s$ is weakly successful. Since it reflects down, it is successful and $\leap{\s^+}$ is the restriction of $\lea$ to $\Ksatp{\lambda^+}_{\lambda^+}$. By adapting the proof of \cite[III.1.5(3)]{shelahaecbook} (see \cite[2.14]{counterexample-frame-afml}), we get that $\s$ is $\goodp$.
  \item \underline{(\ref{reflect-equiv-2}) implies (\ref{reflect-equiv-3})}: By Lemma \ref{goodp-decent}.
  \item \underline{(\ref{reflect-equiv-3}) implies (\ref{reflect-equiv-1})}: By definition of successful, $\s$ is weakly successful, so by Fact \ref{weakly-succ-extension} extends to a unique type-full good $(\le \lambda, \lambda)$-frame. Since $\s$ is decent, we have that the ordering $\preceq_{\lambda^+}^{\otimes}$ from \cite[8.1.2]{jrsh875} is the same as $\lea$. By assumption $\s$ is successful, so all of the equivalent conditions of \cite[9.1.13]{jrsh875} are false, so in particular for $M, N \in \K_{\s^+}$, $M \lea N$ implies $M \leap{\s^+} N$. Therefore $\s$ reflects down.
  \item \underline{(\ref{reflect-equiv-1}) implies (\ref{reflect-equiv-4})}: we have already argued that $\s$ is successful, and by definition $\leap{\s^+}$ is just the restriction of $\lea$, i.e.\ $\K_{\s^+} = \Ksatp{\lambda^+}_{\lambda^+}$. Now weak tameness follows from \cite[III.1.10]{shelahaecbook} (a similar argument already appears in \cite{superior-aec}) and amalgamation is because by Fact \ref{good-ext-thm} $\s^+$ is a good $\lambda^+$-frame.
  \item\underline{(\ref{reflect-equiv-4}) implies (\ref{reflect-equiv-1})}: By \cite[7.15]{jarden-tameness-apal}.
  \end{itemize}
\end{proof}

The aim of this section is to add another condition to Fact \ref{reflect-equiv-fact}: the existence of a good $\lambda^+$-frame on $\Ksatp{\lambda^+}_{\lambda^+}$. Toward this, we will use the following ordering on pairs of saturated models, introduced by Jarden \cite[7.5]{jarden-tameness-apal}. The idea is as follows: call a pair $(M, N)$ of saturated models in $\K_{\lambda^+}$ \emph{bad} if $M \lea N$ but $M \not \leap{\s^+} N$. Jarden's ordering of these pairs works in such a way that higher pairs interact non-trivially with lower pairs. Jarden showed that there are maximal such pairs, and we will show that if there is a good $\lambda^+$-frame then on the other hand it is always possible to extend these pairs. This will be a contradiction showing that there are no bad pairs after all.

\begin{defin}
  Let $\s$ be a weakly successful categorical good $\lambda$-frame generating an AEC $\K$. For pairs $(M, N), (M', N')$ in $\K_{\s^+}$ with $M \lea N$, $M' \lea N'$, we write $(M, N) \tlt (M', N')$ if:

  \begin{enumerate}
  \item $M \leap{\s^+} M'$.
  \item $N \lea N'$.
  \item $M' \cap N \neq M$.
  \end{enumerate}

  We write $(M, N) \tleq (M', N')$ if $(M, N) \tlt (M', N')$ or $(M, N) = (M', N')$.
\end{defin}

\begin{fact}\label{smoothness-contradict-fact}
  Let $\s$ be a \emph{decent} weakly successful categorical good $\lambda$-frame generating an AEC $\K$.

  \begin{enumerate}
  \item\label{smoothness-contradict-1} For any $M \lea N$ both in $\K_{\s^+}$, there exists a pair $(M', N')$ such that $(M, N) \tleq (M', N')$ and $(M', N')$ is $\tlt$-maximal.
  \item\label{smoothness-contradict-2} \cite[9.1.13]{jrsh875} $M \lea N$ are both in $\K_{\s^+}$, and $M \not \leap{\s^+} N$, then there exists $\seq{M_i : i \le \lambda^+}$ increasing continuous in $\K_{\s^+}$ such that $M_{\lambda^+} = M$ and $M_i \leap{\s^+} N$ for all $i < \lambda^+$.
  \end{enumerate}
\end{fact}
\begin{proof} \
  \begin{enumerate}
  \item By the proof of \cite[7.8]{jarden-tameness-apal}. We use decency here to make sure (via Theorem \ref{decent-thm}) that $\Ksatp{\lambda^+}_{\lambda^+}$ is an AEC in $\lambda^+$, hence closed under unions of short chains (Jarden works with the whole of $\K_{\lambda^+}$).
  \item The definition of decent says that the ordering $\preceq_{\lambda^+}^{\otimes}$ from \cite[8.1.2]{jrsh875} is the same as the usual ordering $\lea$ on $\Ksatp{\lambda^+}_{\lambda^+}$. Thus if $M \lea N$ are both in $\K_{\s^+}$ but $M \not \leap{\s^+} N$, then condition (2) in the equivalence proven in \cite[9.1.13]{jrsh875} holds, hence condition (3) there also holds, which is what we wanted to prove.
  \end{enumerate}
\end{proof}

We have arrived to the main theorem of this section:

\begin{thm}\label{main-thm}
  Let $\s$ be a weakly successful categorical good $\lambda$-frame generating an AEC $\K$. If there is a good $\lambda^+$-frame on $\Ksatp{\lambda^+}_{\lambda^+}$, then $\s$ reflects down.
\end{thm}
\begin{proof}
  Let $\ts$ be the good $\lambda^+$-frame on $\Ksatp{\lambda^+}_{\lambda^+}$. Note that since $\ts$ is good, $\Ksatp{\lambda^+}_{\lambda^+}$ must be an AEC in $\lambda^+$, i.e.\ unions of chains of $\lambda^+$-saturated models are $\lambda^+$-saturated, so by Fact \ref{sl-easy} $\K$ has a superlimit in $\lambda^+$ which by Theorem \ref{decent-thm} implies that $\s$ is decent. Therefore we will later be able to apply Fact \ref{smoothness-contradict-fact}. We first show:

  \underline{Claim}: If $M \lea N$ are both in $\K_{\s^+}$, $a \in |N| \backslash |M|$ and $\gtpp{\K} (a / M ; N)$ is realized in a $\leap{\s^+}$-extension of $M$, then $(M, N)$ is not $\tlt$-maximal.

  \underline{Proof of Claim}: Let $p := \gtpp{\K} (a / M; N)$. Let $N'$ be such that $M \leap{\s^+} N'$ and $p$ is realized by $b \in N'$ (i.e.\ $p = \gtpp{\K} (b / M; N')$). By amalgamation in $\Ksatp{\lambda^+}_{\lambda^+}$ (which holds since we are assuming there is a good frame on this class) and the fact that $\gtpp{\K} (a / M; N) = \gtpp{\K} (b / M; N')$, there exists $N'' \in \Ksatp{\lambda^+}_{\lambda^+}$ and $f: N' \xrightarrow[M]{} N''$ with $N \lea N''$ such that $f (b) = a$. Consider the pair $(f[N'], N'')$. Since $\leap{\s^+}$ is invariant under isomorphisms and $M \leap{\s^+} N'$, $M \leap{\s^+} f[N']$. Moreover, $a \in |N| \backslash |M|$, so $b \in |N'| \backslash |M|$, and so $a = f (b) \in |f[N']| \backslash |M|$. Since we also have that $a \in |N|$, this implies that $f[N'] \cap N \neq  M$, so $(M, N) \tlt (f[N'], N'')$, hence $(M, N)$ is not $\tlt$-maximal. $\dagger_{\text{Claim}}$
  
  Let $M, N \in \K_{\s^+}$ be such that $M \lea N$. We have to show that $M \leap{\s^+} N$. Suppose not. By Fact \ref{smoothness-contradict-fact}(\ref{smoothness-contradict-1}), there exists $(M', N')$ such that $(M, N) \tleq (M', N')$ and $(M', N')$ is $\tlt$-maximal. Observe that $M' \not \leap{\s^+} N'$: if $M' \leap{\s^+} N'$, then since also $M \leap{\s^+} M'$ and $\K_{\s^+}$ is an abstract class (Fact \ref{weakly-succ-facts}(\ref{weakly-succ-ac})), we would have that $M \leap{\s^+} N'$ so by Fact \ref{weakly-succ-facts}(\ref{weakly-succ-monot}), $M \leap{\s^+} N$, a contradiction. By Fact \ref{smoothness-contradict-fact}(\ref{smoothness-contradict-2}), there exists $\seq{M_i : i \le \lambda^+}$ increasing continuous in $\K_{\s^+}$ such that $M_{\lambda^+} = M'$ and $M_i \leap{\s^+} N'$ for all $i < \lambda^+$. Since $M' \not \leap{\s^+} N'$, we have in particular that $M' \neq N'$. By density of basic types, let $a \in N' \backslash M'$ and let $p := \gtpp{\K} (a / M'; N')$ be basic. By local character in $\ts$ (the good $\lambda^+$-frame on $\Ksatp{\lambda^+}_{\lambda^+}$), there is $i < \lambda^+$ such that $p$ does not $\ts$-fork over $M_i$. By conjugation in $\ts$ (Fact \ref{conj-fact}), this means that $p \rest M_i$ and $p$ are isomorphic. However since $M_i \leap{\s^+} N'$, $p \rest M_i$ is realized inside some $\leap{\s^+}$-extension of $M_i$, hence $p$ is also realized inside some $\leap{\s^+}$-extension of $M'$. Together with the claim, this contradicts the $\tlt$-maximality of $(M', N')$.
\end{proof}
\begin{remark}
  We are not using all the properties of the good $\lambda^+$-frame. In particular, it suffices that local character holds for chains of length $\lambda^+$.
\end{remark}

We obtain a new characterization of when a weakly successful frame reflects down. We emphasize that only (\ref{main-cor-3}) implies (\ref{main-cor-1}) below is new. (\ref{main-cor-1}) implies (\ref{main-cor-2}) and (\ref{main-cor-1}) implies (\ref{main-cor-3}) are due to Shelah while (\ref{main-cor-2}) implies (\ref{main-cor-1}) is due to Jarden.

\begin{cor}\label{main-cor}
  Let $\s$ be a weakly successful categorical good $\lambda$-frame generating the AEC $\K$. The following are equivalent:

  \begin{enumerate}
  \item\label{main-cor-1} $\s$ is successful and decent.
  \item\label{main-cor-2} $\K$ is $(\lambda, \lambda^+)$-weakly tame and $\Ksatp{\lambda^+}_{\lambda^+}$ has amalgamation.
  \item\label{main-cor-3} There is a  good $\lambda^+$-frame on $\Ksatp{\lambda^+}_{\lambda^+}$.
  \end{enumerate}

  Moreover, if it exists then the good $\lambda^+$-frame $\ts$ on $\Ksatp{\lambda^+}_{\lambda^+}$ is unique, can be enlarged to be type-full, and coincides (on the appropriate class of basic types) with $\s^+$ and $\s_{\lambda^+} \rest \Ksatp{\lambda^+}_{\lambda^+}$ (see Definition \ref{s-lambdap-def}).
\end{cor}
\begin{proof} \
  \begin{itemize}
    \item \underline{(\ref{main-cor-1}) is equivalent to (\ref{main-cor-2})}: By Fact \ref{reflect-equiv-fact}.
    \item \underline{(\ref{main-cor-1}) implies (\ref{main-cor-3})}: By Facts \ref{good-ext-thm} and \ref{reflect-equiv-fact}.
    \item \underline{(\ref{main-cor-3}) implies (\ref{main-cor-1})}: By Theorem \ref{main-thm} and Fact \ref{reflect-equiv-fact}.
  \end{itemize}

  For the moreover part, assume that $\ts$ is a good $\lambda^+$-frame on $\Ksatp{\lambda^+}_{\lambda^+}$. By the proof of (\ref{main-cor-1}) implies (\ref{main-cor-3}), $\s^+$ is in fact a good $\lambda^+$-frame on $\Ksatp{\lambda^+}_{\lambda^+}$. By weak tameness and Theorem \ref{frame-ext-weak}, $\s_{\lambda^+} \rest \Ksatp{\lambda^+}_{\lambda^+}$ is also a good $\lambda^+$-frame on $\Ksatp{\lambda^+}_{\lambda^+}$. Without loss of generality (Fact \ref{reflect-equiv-fact}(\ref{reflect-equiv-1})), $\s$ is type-full so by the moreover part of Theorem \ref{frame-ext-weak}, $\s_{\lambda^+} \rest \Ksatp{\lambda^+}_{\lambda^+}$ is also type-full, and by canonicity (Fact \ref{canon-fact}) is the desired extension of $\ts$.
\end{proof}

We can combine our result with the weak GCH to obtain weak tameness from two successive good frames. This gives a converse to Theorem \ref{frame-ext-weak}.

\begin{cor}\label{main-cor-gch}
  Let $\K$ be an AEC, let $\lambda \ge \LS (\K)$ and assume that $2^{\lambda} < 2^{\lambda^+}$. Let $\s$ be a categorical good $\lambda$-frame on $\K_{\lambda}$. The following are equivalent:

  \begin{enumerate}
  \item\label{gch-1} There is a $\goodp$ $\lambda^+$-frame on $\Ksatp{\lambda^+}_{\lambda^+}$.
  \item\label{gch-2} There is a good $\lambda^+$-frame on $\Ksatp{\lambda^+}_{\lambda^+}$.
  \item\label{gch-3} $\K$ is $(\lambda, \lambda^+)$-weakly tame, $\s$ is decent, and $\Ksatp{\lambda^+}_{\lambda^+}$ has amalgamation.
  \item\label{gch-4} $\K$ is $(\lambda, \lambda^+)$-weakly tame and for every saturated $M \in \K_{\lambda^+}$ there is $N \in \K_{\lambda^+}$ universal over $M$.
  \item\label{gch-5} $\s$ is successful and $\goodp$.
  \end{enumerate}
\end{cor}
\begin{proof} \
  \begin{itemize}
  \item \underline{(\ref{gch-1}) implies (\ref{gch-2})}: Trivial.
  \item \underline{(\ref{gch-2}) implies (\ref{gch-3})}: By Fact \ref{weakly-succ-gch}, $\s$ is weakly successful. Now apply Corollary \ref{main-cor}.
  \item \underline{(\ref{gch-3}) implies (\ref{gch-2})}: By Theorem \ref{frame-ext-weak}.
  \item \underline{(\ref{gch-2}) implies (\ref{gch-4})}: Assume (\ref{gch-2}). Then by definition of a good $\lambda^+$-frame, for every saturated $M \in \K_{\lambda^+}$ there is $N \in \K_{\lambda^+}$ universal over $M$. Further we proved already that (\ref{gch-3}) holds. Therefore $\K$ is $(\lambda, \lambda^+)$-weakly tame.
  \item \underline{(\ref{gch-4}) implies (\ref{gch-5})}: By Fact \ref{weakly-succ-gch}, $\s$ is weakly successful. Now apply (\ref{reflect-equiv-4}) implies (\ref{reflect-equiv-2}) in Fact \ref{reflect-equiv-fact}.
  \item \underline{(\ref{gch-5}) implies (\ref{gch-1})}: By Fact \ref{good-ext-thm} and since $\s$ reflects down (Fact \ref{reflect-equiv-fact}).
  \end{itemize}
\end{proof}

\section{From weak to strong tameness}

Assuming amalgamation in $\lambda^+$, we are able to conclude $(\lambda, \lambda^+)$-tameness in (\ref{main-cor-2}) of Corollary \ref{main-cor}. To prove this, we recall that the definition of $\s^+$ (Definition \ref{succ-def}) can be extended to all of $\K_\lambda$ in the following way:

\begin{defin}\label{star-def}
  Let $\s = (\K_{\s}, \nf)$ be a good $(\le \lambda, \lambda)$-frame generating an AEC $\K$. We define a pair $\s^\ast = (\K_{\s^\ast}, \nf_{\s^\ast})$ as follows:

  \begin{enumerate}
  \item $\K_{\s^\ast} = (K_{\s^\ast}, \leap{\s^\ast})$, where:
    \begin{enumerate}
    \item $K_{\s^\ast} = K_{\lambda^+}$.
    \item For $M, N \in K_{\s^\ast}$, $M \leap{\s^\ast} N$ holds if and only if there exists increasing continuous chains $\seq{M_i : i < \lambda^+}$, $\seq{N_i : i < \lambda^+}$ in $\K_\lambda$ such that:

  \begin{enumerate}
  \item $M = \bigcup_{i < \lambda^+} M_i$.
  \item $N = \bigcup_{i < \lambda^+} N_i$.
  \item For all $i < j < \lambda^+$, $\nfs{M_i}{N_i}{M_{j}}{N_{j}}$.
  \end{enumerate}
    \end{enumerate}
  \item For $M_0 \leap{\s^\ast} M \leap{\s^\ast} N$ and $a \in N$, $\nf_{\s^\ast} (M_0, a, M, N)$ holds if and only if there exists $M_0' \in \PKp{\lambda} (M_0)$ such that for all $M' \in \PKp{\lambda} (M)$ and all $N' \in \PKp{\lambda} (N)$, if $M_0' \lea M' \lea N'$ and $a \in N'$, then $\nf_{\s} (M_0', a, M', N')$.
  \end{enumerate}

  For a weakly successful categorical good $\lambda$-frame $\s$, we define $\s^\ast := \ts^\ast$, where $\ts$ is the unique extension of $\s$ to a good $(\le \lambda, \lambda)$-frame (Fact \ref{weakly-succ-extension}).
\end{defin}

The parts of \cite{jrsh875} referenced by Fact \ref{weakly-succ-facts} apply to the whole of $\K_{\lambda^+}$ (i.e.\ there is no need to restrict to saturated models). Thus we have:

\begin{fact}\label{weakly-succ-facts-2}
  Let $\s$ be a weakly successful categorical good $\lambda$-frame.

  \begin{enumerate}
  \item\label{weakly-succ-ap-2} \cite[6.1.5,6.1.6]{jrsh875} $\s^\ast$ is a $\lambda^+$-frame and $\K_{\s^\ast}$ has amalgamation.
  \item\label{weakly-succ-monot-2} \cite[6.1.6(c)]{jrsh875} if $M_0, M_1, M_2 \in \K_{\s^\ast}$ are such that $M_0 \leap{\s^\ast} M_2$ and $M_0 \lea M_1 \lea M_2$, then $M_0 \leap{\s^\ast} M_1$.
  \end{enumerate}
\end{fact}

We will also use that every model of $\K_{\lambda^+}$ has a saturated $\leap{\s^\ast}$-extension.

\begin{fact}[7.1.10, 7.1.12(a) in \cite{jrsh875}]\label{nf-sat-ext}
  Let $\s$ be a weakly successful categorical good $\lambda$-frame. For any $M \in \K_{\s^\ast}$, there exists $N \in \K_{\s^+}$ such that $M \leap{\s^\ast} N$. 
\end{fact}

Similarly to Definition \ref{reflect-def}, we name what it means for $\leap{\s^\ast}$ to be trivial:

\begin{defin}\label{reflect-def-2}
  Let $\s$ be a good $(\le \lambda, \lambda)$-frame generating an AEC $\K$. We say that $\s$ \emph{strongly reflects down} if $M \lea N$ implies $M \leap{\s^\ast} N$ for all $M, N \in \K_{\s^\ast}$. We say that a weakly successful good $\lambda$-frame \emph{strongly reflects down} if its extension to a good $(\le \lambda, \lambda)$-frame strongly reflects down, see Fact \ref{weakly-succ-extension}.
\end{defin}

We establish the following criteria for strongly reflecting down:

\begin{thm}\label{strong-reflection-thm}
  Let $\s$ be a weakly successful categorical good $\lambda$-frame generating an AEC $\K$. The following are equivalent:
  \begin{enumerate}
  \item\label{strong-refl-1} $\s$ reflects down and $\K$ has amalgamation in $\lambda^+$.
  \item\label{strong-refl-2} $\s$ strongly reflects down.
  \end{enumerate}
\end{thm}
\begin{proof}
  (\ref{strong-refl-2}) implies (\ref{strong-refl-1}) is Fact \ref{weakly-succ-facts-2}(\ref{weakly-succ-ap-2}). Assume now that (\ref{strong-refl-1}) holds. Let $M, N \in \K_{\s^\ast}$ be such that $M \lea N$. By Fact \ref{nf-sat-ext}, there exists $M' \in \K_{\s^+}$ such that $M \leap{\s^\ast} M'$. By Facts \ref{good-ext-thm} and \ref{reflect-equiv-fact}, $\s^+$ is a good $\lambda^+$-frame. In particular, there exists $N' \in \K_{\s^+}$ such that $M' \leap{s^+} N'$ and $N'$ is universal over $M'$ in $\K_{\s^+}$. Since $\K$ has amalgamation in $\lambda^+$ and $\leap{\s^+}$ is the same as $\lea$ by definition of reflecting down, $N'$ is universal over $M$. Moreover, $M \leap{\s^\ast} N'$ by transitivity of $\leap{\s^\ast}$. Now let $f: N \xrightarrow[M]{} N'$ by a $\K$-embedding. By Fact \ref{weakly-succ-facts-2}(\ref{weakly-succ-monot-2}), $M \leap{\s^\ast} f[N]$, so by invariance $M \leap{\s^\ast} N$, as desired.
\end{proof}

It is known that strongly reflecting down gives tameness. This is \cite[7.1.17(b)]{jrsh875}. There is a mistake in the statement given by Jarden and Shelah ($\preceq$ there should be $\preceq_{\text{NF}}$) but the proof is still correct.

\begin{fact}\label{strong-tameness-fact}
  Let $\s$ be a weakly successful categorical good $\lambda$-frame generating an AEC $\K$. If $\s$ strongly reflects down, then $\K$ is $(\lambda, \lambda^+)$-tame.
\end{fact}

We obtain that tameness is equivalent to being able to extend the frame. Note that only (\ref{strong-tameness-15}) implies (\ref{strong-tameness-1}) and (\ref{strong-tameness-3}) implies (\ref{strong-tameness-1}) are new.

\begin{cor}\label{strong-tameness-cor}
  Let $\s$ be a weakly successful categorical good $\lambda$-frame generating an AEC $\K$. The following are equivalent:

  \begin{enumerate}
  \item\label{strong-tameness-1} $\s$ strongly reflects down.
  \item\label{strong-tameness-15} $\K$ has amalgamation in $\lambda^+$ and is $(\lambda, \lambda^+)$-weakly tame.    
  \item\label{strong-tameness-2} $\K$ has amalgamation in $\lambda^+$ and is $(\lambda, \lambda^+)$-tame.
  \item\label{strong-tameness-3} There is a good $\lambda^+$-frame on $\K_{\lambda^+}$.
  \end{enumerate}
\end{cor}
\begin{proof} \
  \begin{itemize}
  \item \underline{(\ref{strong-tameness-1}) implies (\ref{strong-tameness-2})}: By Facts \ref{weakly-succ-facts-2}(\ref{weakly-succ-ap-2}) and \ref{strong-tameness-fact}.
  \item \underline{(\ref{strong-tameness-2}) implies (\ref{strong-tameness-15})}: Trivial.
  \item \underline{(\ref{strong-tameness-15}) implies (\ref{strong-tameness-1})}: By Corollary \ref{main-cor}, $\s$ is successful and decent. By Fact \ref{reflect-equiv-fact}, $\s$ reflects down. By Theorem \ref{strong-reflection-thm}, $\s$ strongly reflects down.
  \item \underline{(\ref{strong-tameness-2}) implies (\ref{strong-tameness-3})}: By Fact \ref{ext-frame}.
  \item \underline{(\ref{strong-tameness-3}) implies (\ref{strong-tameness-1})}: Assume that (\ref{strong-tameness-3}) holds. By Lemma \ref{decent-from-frame}, $\s$ is decent so by Theorem \ref{decent-thm}, $\K$ has a superlimit in $\lambda^+$. Thus by Fact \ref{sl-easy} we can restrict the good $\lambda^+$-frame on $\K_{\lambda^+}$ to the class $\Ksatp{\lambda^+}_{\lambda^+}$ and still obtain a good $\lambda^+$-frame. By Theorem \ref{main-thm}, $\s$ reflects down. Since we are assuming there is a good $\lambda^+$-frame on $\K_{\lambda^+}$, $\K$ has amalgamation in $\lambda^+$, so by Theorem \ref{strong-reflection-thm}, $\s$ strongly reflects down.
  \end{itemize}
\end{proof}

Assuming $2^{\lambda} < 2^{\lambda^+}$, we do not need to assume that $\s$ is weakly successful. We do not repeat all the equivalences of Corollary \ref{strong-tameness-cor} and only state our main result:

\begin{cor}\label{main-cor-simplified}
  Let $\K$ be an AEC, let $\lambda \ge \LS (\K)$ and assume that $2^{\lambda} < 2^{\lambda^+}$. If there is a categorical good $\lambda$-frame $\s$ on $\K_{\lambda}$ and a good $\lambda^+$-frame $\ts$ on $\K_{\lambda^+}$, then $\K$ is $(\lambda, \lambda^+)$-tame. Moreover, $\ts \rest \Ksatp{\lambda^+}_{\lambda^+} = \s_{\lambda^+} \rest \Ksatp{\lambda^+}_{\lambda^+}$ (see Definition \ref{s-lambdap-def}).
\end{cor}
\begin{proof}
  Let $\s$ be a categorical good $\lambda$-frame on $\K_\lambda$. By Fact \ref{weakly-succ-gch}, $\s$ is weakly successful. Now apply Corollary \ref{strong-tameness-cor}. The moreover part follows from the moreover part of Corollary \ref{main-cor}.
\end{proof}

\section{On categoricity in two successive cardinals}\label{categ-sec}

Grossberg and VanDieren have shown \cite[6.3]{tamenessthree} that in a $\lambda$-tame AEC with amalgamation and no maximal models, categoricity in $\lambda$ and $\lambda^+$ imply categoricity in all $\mu \ge \lambda$. In \cite{downward-categ-tame-apal}, the author gave a more local conclusion as well as a more abstract proof using good frames. Here, we give a converse: assuming the weak GCH, if we can prove categoricity in $\lambda^{++}$ from categoricity in $\lambda$ and $\lambda^+$, then we must have some tameness. This follows from combining results of Shelah\footnote{Let us sketch how (of course Corollary \ref{categ-cor} gives another proof but it uses the results of the present paper). By Fact \ref{shelah-categ-frame}(\ref{frame-4}), there is a successful good $\lambda$-frame $\s$ on $\K_\lambda$. In fact, the basic types of this good frames are the minimal types. Thus $\s$ is $\goodp$ (as nonforking corresponds to disjoint amalgamation). It then follows from \cite[III.1.11]{shelahaecbook} and categoricity that $\K$ is $(\lambda, \lambda^+)$-tame.} but seems not to have been noticed before. Only (\ref{categ-05}) implies (\ref{categ-1}) below really uses the results of this paper.

\begin{cor}\label{categ-cor}
  Let $\K$ be an AEC and let $\lambda \ge \LS (\K)$. Assume $2^{\lambda} < 2^{\lambda^+} < 2^{\lambda^{++}}$ and assume that $\K$ is categorical in $\lambda$ and $\lambda^{+}$. The following are equivalent:

  \begin{enumerate}

  \item\label{categ-0} There is a successful good $\lambda$-frame on $\K_\lambda$.
  \item\label{categ-05} There is a good $\lambda$-frame on $\K_{\lambda}$ and a good $\lambda^+$-frame on $\K_{\lambda^+}$.
  \item\label{categ-1} $\K$ is stable in $\lambda$, $\K$ is $(\lambda, \lambda^+)$-tame, $\K$ has amalgamation in $\lambda^+$, and $\K_{\lambda^{++}} \neq \emptyset$.
  \item\label{categ-2} $\K$ is categorical in $\lambda^{++}$.
  \item\label{categ-3} $1 \le \Ii (\K, \lambda^{++}) < \mu_{\text{unif}} (\lambda^{++}, 2^{\lambda^+})$.
  \item\label{categ-4} $\K$ is stable in $\lambda$, $\K$ is stable in $\lambda^+$, and $1 \le \Ii (\K, \lambda^{++}) < 2^{\lambda^{++}}$.
  \end{enumerate}
\end{cor}

On $\mu_{\text{unif}}$, see \cite[VII.0.4]{shelahaecbook2} for a definition and \cite[VII.9.4]{shelahaecbook2} for what is known. It seems that for all practical purposes the reader can take $\mu_{\text{unif}} (\lambda^{++}, 2^{\lambda^+})$ to mean $2^{\lambda^{++}}$.

Note that if the AEC $\K$ of Corollary \ref{categ-cor} has arbitrarily large models, then stability in $\lambda$ would follow from categoricity in $\lambda^+$ \cite[I.1.7(a)]{sh394}. Thus we obtain that $\K$ is categorical in $\lambda^{++}$ if and only if $\K$ is $(\lambda, \lambda^+)$-tame and has amalgamation in $\lambda^{+}$.

To prove Corollary \ref{categ-cor}, we will use several facts:

\begin{fact}[I.3.8 in \cite{shelahaecbook}]\label{ap-categ}
  Let $\K$ be an AEC and let $\lambda \ge \LS (\K)$. Assume $2^{\lambda} < 2^{\lambda^+}$. If $\K$ is categorical in $\lambda$ and $\Ii (\K, \lambda^+) < 2^{\lambda^+}$, then $\K$ has amalgamation in $\lambda$.
\end{fact}
\begin{fact}\label{shelah-categ-frame}
  Let $\K$ be an AEC and let $\lambda \ge \LS (\K)$. Assume $2^{\lambda} < 2^{\lambda^+}$ and assume that $\K$ is categorical in both $\lambda$ and $\lambda^+$.

  \begin{enumerate}
  \item\label{frame-1} If $\K_{\lambda^{++}} \neq \emptyset$ and $\K$ is stable in $\lambda$, then there is an almost good $\lambda$-frame (see \cite[VI.8.2]{shelahaecbook2}) on $\K_\lambda$.
  \item\label{frame-2} If $\K_{\lambda^{++}} \neq \emptyset$, $\K$ has amalgamation in $\lambda^+$, $\K$ is stable in $\lambda$, and $\K$ is stable in $\lambda^{+}$, then there is a weakly successful good $\lambda$-frame on $\K_\lambda$.
  \item\label{frame-3} If $2^{\lambda^+} < 2^{\lambda^{++}}$, $1 \le \Ii (\K, \lambda^{++}) < 2^{\lambda^{++}}$, $\K$ is stable in $\lambda$, and $\K$ is stable in $\lambda^+$, then there is a successful good $\lambda$-frame on $\K_\lambda$.
  \item\label{frame-4} If $2^{\lambda} < 2^{\lambda^+} < 2^{\lambda^{++}}$ and $1 \le \Ii (\K, \lambda^{++}) < \mu_{\text{unif}} (\lambda^{++}, 2^{\lambda^+})$, then there is a successful good $\lambda$-frame on $\K_\lambda$.
  \end{enumerate}
\end{fact}
\begin{proof} \
  \begin{enumerate}
  \item By Fact \ref{ap-categ}, $\K$ has amalgamation in $\lambda$. We check that the hypotheses of \cite[VI.8.1(2)]{shelahaecbook2} are satisfied. The only ones that we are not explicitly assuming are:
    \begin{enumerate}
    \item The extension property in $\K_\lambda$, i.e.\ for every $M \lea N$ both in $\K_\lambda$ and every $p \in \gS (M)$, if $p$ is not algebraic (i.e.\ not realized inside $M$), then $p$ has a \emph{nonalgebraic} extension to $\gS (N)$: holds by \cite[VI.2.23(1)]{shelahaecbook2}.
    \item The existence of an inevitable type in $\K_\lambda$: holds by \cite[VI.5.3(1)]{shelahaecbook2} and the density of minimal types (see the proof of $(\ast)_5$ in \cite[II.2.7]{sh394}).
    \end{enumerate}
  \item By (\ref{frame-1}), there is an almost good $\lambda$-frame $\s$ on $\K_\lambda$. The proof of \cite[E.8]{downward-categ-tame-apal} goes through even for almost good $\lambda$-frames and gives that $\s$ is weakly successful (note that Shelah's proof of Fact \ref{conj-fact} still goes through in almost good $\lambda$-frames). By \cite[4.3]{jash940-v1}, $\s$ is in fact a good $\lambda$-frame.
  \item By Fact \ref{ap-categ}, $\K$ has amalgamation in $\lambda$ and $\lambda^+$. By (\ref{frame-2}), there is a weakly successful good $\lambda$-frame $\s$ on $\K_{\lambda}$. By Fact \ref{succ-nm}, $\s$ is successful.
  \item By \cite[VI.8.1]{shelahaecbook2}, there is an almost good $\lambda$-frame $\s$ on $\K_\lambda$. By \cite[VII.6.17]{shelahaecbook2}, $\s$ has existence for a certain relative of uniqueness triples. Thus by \cite[VII.7.19(2)]{shelahaecbook2}, $\s$ is actually a good $\lambda$-frame. By \cite[II.5.11]{shelahaecbook}, $\s$ is weakly successful. By Fact \ref{succ-nm}, $\s$ is successful.
  \end{enumerate}
\end{proof}

\begin{fact}\label{upward-categ}
  Let $\K$ be an AEC and let $\lambda \ge \LS (\K)$. If $\K$ is categorical in both $\lambda$ and $\lambda^+$ and there is a successful good $\lambda$-frame on $\K_\lambda$, then $\K$ is categorical in $\lambda^{++}$.
\end{fact}
\begin{proof}
  Let $\s$ be a successful good $\lambda$-frame on $\K_\lambda$. Since $\K$ is categorical in $\lambda^+$, there is a superlimit in $\lambda^+$, hence by Theorem \ref{decent-thm} $\s$ is decent, hence by Fact \ref{reflect-equiv-fact} is $\goodp$. Now combine \cite[III.2.10(2)]{shelahaecbook} with \cite[III.2.11(1)]{shelahaecbook}.
\end{proof}

\begin{proof}[Proof of Corollary \ref{categ-cor}]
  By categoricity in $\lambda$ and $\lambda^+$ all good $\lambda$-frames on $\K_\lambda$ are categorical, $\Ksatp{\lambda^+}_{\lambda^+} = \K_{\lambda^+}$, and $(\lambda, \lambda^+)$-weak tameness is the same as $(\lambda, \lambda^+)$-tameness. By Theorem \ref{decent-thm}, any good $\lambda$-frame on $\K_\lambda$ is decent. We will use this without comments.
  
  \begin{itemize}
  \item \underline{(\ref{categ-0}) implies (\ref{categ-05})}:    By Fact \ref{reflect-equiv-fact}, the definition of reflecting down (Definition \ref{reflect-def}) and the definition of $\s^+$ (Definition \ref{succ-def}), $\K_{\s^+} = \Ksatp{\lambda^+}_{\lambda^+} = \K_{\lambda^+}$, so the result follows from Fact \ref{good-ext-thm}.
  \item \underline{(\ref{categ-05}) implies (\ref{categ-1})}: By definition of a good $\lambda$-frame, $\K$ is stable in $\lambda$ and by definition of a good $\lambda^+$-frame $\K$ has amalgamation in $\lambda^+$. Since a good $\lambda^+$-frame has no maximal models in $\lambda^+$, $\K_{\lambda^{++}} \neq \emptyset$. By Corollary \ref{main-cor-gch}, $\K$ is $(\lambda, \lambda^+)$-tame.
    \item \underline{(\ref{categ-1}) implies (\ref{categ-0})}: By Fact \ref{ap-categ}, $\K$ has amalgamation in $\lambda$. By Fact \ref{shelah-categ-frame}, there is an almost good $\lambda$-frame on $\K_\lambda$. By the proof of Fact \ref{ext-frame}, $\K$ is stable in $\lambda^+$. By Fact \ref{shelah-categ-frame}, there is a weakly successful good $\lambda$-frame $\s$ on $\K_\lambda$. By Fact \ref{reflect-equiv-fact}, $\s$ is successful.
    \item \underline{(\ref{categ-0}) implies (\ref{categ-2})}: By Fact \ref{upward-categ}.
    \item \underline{(\ref{categ-2}) implies (\ref{categ-3})}: Trivial.
    \item   \underline{(\ref{categ-3}) implies (\ref{categ-0})}: By Fact \ref{shelah-categ-frame}.
    \item \underline{(\ref{categ-3}) implies (\ref{categ-4})}: (\ref{categ-3}) trivially implies that $1 \le \Ii (\K, \lambda^{++}) < 2^{\lambda^{++}}$. We have also seen that (\ref{categ-3}) implies (\ref{categ-0}) implies (\ref{categ-05}) which by definition implies stability in $\lambda$ and $\lambda^+$.
    \item \underline{(\ref{categ-4}) implies (\ref{categ-0})}: By Fact \ref{shelah-categ-frame}.
      \end{itemize}
\end{proof}

\bibliographystyle{amsalpha}
\bibliography{local-ss}

\end{document}